\documentclass[a4paper, 12pt, oneside, notitlepage]{amsart}

\usepackage{amsmath,amscd}
\usepackage{amssymb}
\usepackage{amsthm}
\usepackage{comment}
\usepackage[usenames,dvipsnames]{xcolor}

\usepackage{mathrsfs}
\usepackage[colorlinks=true,linkcolor=Red,citecolor=PineGreen]{hyperref}

\usepackage{thmtools}
\usepackage[margin=3cm]{geometry}
\usepackage{enumitem}

\newcommand{\WF}{{\rm WF}}
\theoremstyle{plain}
\newtheorem{thm}{Theorem}[section]
\newtheorem{prop}[thm]{Proposition}
\newtheorem{lem}[prop]{Lemma}
\newtheorem{cor}[prop]{Corollary}

\newtheorem{rmk}[prop]{Remark}

\newtheorem*{thm*}{Theorem}

\numberwithin{equation}{section}
\newcommand{\mc}{\mathcal}
\newcommand{\1}{\mathbf{1}}
\newcommand {\R} {\mathbb{R}} \newcommand {\Z} {\mathbb{Z}}
\newcommand {\T} {\mathbb{T}} \newcommand {\N} {\mathbb{N}}
\newcommand {\C} {\mathbb{C}}

\newcommand{\B}{\mathbb{B}}

\newcommand {\D} {{\mathcal D}}
\newcommand {\dd} {\mathrm{d}}
\newcommand{\eps}{\varepsilon}

\newcommand {\supp} {\text{supp}}

\newcommand{\bV}{{\mathbf V}}
\newcommand{\bH}{{\mathbf H}}
 
\newcommand{\mS}{\mathbb{S}}                   
 
\definecolor{armygreen}{rgb}{0.29, 0.33, 0.13}
\definecolor{ao(english)}{rgb}{0.0, 0.5, 0.0}

\DeclareMathOperator {\sign} {sgn}

\title[Geodesics, L\'evy Flights, Random Searches]{Geodesic L\'evy Flights and Expected Stopping Time for Random Searches}

\author{Yann Chaubet}
\address{Cambridge University, Faculty of Mathematics, Wilberforce Road, Cambridge CB3 0WB, United Kingdom}
\email{yann.chaubet@dpmms.cam.ac.uk	}

\author{Yannick Guedes Bonthonneau}
\address{Institut Galil\'ee, Universit\'e Paris 13, avenue Jean-Baptiste Clément
93430 - Villetaneuse}
\email{bonthonneau@math.univ-paris13.fr}

\author{Thibault Lefeuvre}
\address{Université de Paris and Sorbonne Université, CNRS, IMJ-PRG, F-75006 Paris, France.}
\email{tlefeuvre@imj-prg.fr}

\author{Leo Tzou}
\address{University of Amsterdam, Korteweg-de Vries Institute,  Science Park, 1098XH Amsterdam, Netherlands}
\email{leo.tzou@gmail.com}

\date{\today}

\begin{document}

\begin{abstract} We give an analytic description for the infinitesimal generator constructed in \cite{AppEst} for L\'evy flights on a broad class of closed Riemannian manifolds including all negatively-curved manifolds, the flat torus and the sphere. Various properties of the associated semigroup and the asymptotics of the expected stopping time for L\'evy flight based random searches for small targets, also known as the ``narrow capture problem", are then obtained using our newfound understanding of the infinitesimal generator. {Our study also relates to the \emph{L\'evy flight foraging hypothesis} in the field of biology as we compute the expected time for finding a small target by using the L\'evy flight random search. A similar calculation for Brownian motion on surfaces was done in \cite{nursultanov2022narrow}}.
\end{abstract}

\maketitle



\section{Introduction}



The \emph{L\'evy flight foraging hypothesis} is a well-known hypothesis in the field of biology asserting that animals foraging behaviours should be modelled by L\'evy flights insofar as they may optimize search efficiencies. While this hypothesis has been around for more than twenty years, it is still controversial and subject to many research articles investigating whether Brownian motion or L\'evy flights are optimal search strategies \cite{palyulin2014levy, shlesinger1986growth, viswanathan2011physics, bressloff2013stochastic, benichou2011intermittent, dipierro2022efficiency}. The purpose of this article is to shed a new theoretical light on this question by means of a precise mathematical study.

More precisely, we will investigate the \emph{narrow capture problem} which consists in finding a small target in space for a motion whose law is that of a L\'evy flight. The interesting quantity to understand then is the \emph{expected capture time}, namely, the expected time that a process starting at a given point $p$ will eventually find the target. This small target typically models a prey hunted by a predator whose foraging behaviour is modelled by the L\'evy process. The L\'evy flight foraging hypothesis can then be phrased as follows: \emph{is the expected capture time significantly lower if one uses a search based on L\'evy flights rather than on Brownian motion?}

For bounded domains in the Euclidian space, there are various search strategies based on Brownian motion and in this case an important set of literatures already investigated the expected time of finding small targets \cite{singer2008narrow, singer2006narrow, gomez2015asymptotic, cheviakov2010asymptotic, chen2011asymptotic, ammari2012layer}. However, while the background geometry of many animal foraging behaviours and constraint optimization searches are naturally curved, we have only recently started addressing the question of expected stopping time for Brownian motions on Riemannian manifolds \cite{nursultanov2022narrow, nursultanov2021mean, nursultanov2021narrow}. Thus far, nothing has been done for stopping time for L\'evy flight based searches even in flat geometry. We address this question here for a class of isotropic pure jump L\'evy processes introduced by Applebaum--Estrade \cite{AppEst}. In particular, we investigate the asymptotics of the expected stopping time for a L\'evy flight based random search to find a target the size of a small geodesic ball whose radius converges to zero. 


\subsection{Main result}

\label{subsec:results}

We assume throughout that $(M, g)$ is a smooth closed (that is, compact without boundary) connected $n$-dimensional Riemannian manifold with $n \geqslant 2$. We let $(X_t)_{t \geqslant 0}$ be a cadlag semi-martingale on $M$ which is an isotropic L\'evy process in the sense of \cite{AppEst}, induced by the isostropic L\'evy measure
\begin{equation}\label{eq:levymeasure}
\nu_p (A) = C(n,\alpha)\int_{A} \frac{\dd T_p(v)}{|v|_g^{n+2\alpha}}, \quad A \subset T_pM,\quad  \alpha\in (0,1)
\end{equation}
on each tanget space. Here $T_p$ is the volume form on $T_pM$ induced by the metric $g|_{T_pM}$ and $C(n,\alpha)$ is the constant\footnote{This constant is chosen to be consistent with the definition of the fractional Laplacian on $\R^n$, which is the infinitesmial generator of $2\alpha$-stable isotropic L\'evy processes in Euclidean space.}
\begin{eqnarray}
\label{Cna}
C(n,\alpha) := \frac{4^\alpha \Gamma(n/2+\alpha)}{\pi^{n/2}|\Gamma(-\alpha)|}.
\end{eqnarray}
Fix $p_0 \in M$ and let $B_\varepsilon(p_0)$ be the open geodesic ball of radius $\varepsilon>0$ centered at $p_0$. We define the \emph{expected stopping time} as:
\begin{equation}
\label{eq:deftaueps}
\tau_{\varepsilon} = \inf\left\{ t\geqslant 0 \mid X_t \in \overline{B_\varepsilon(p_0)}\right\} \quad \text{and} \quad u_{\varepsilon}(q) = {\mathbb E}\left(\tau_{\varepsilon}\mid X_0 = q\right),
\end{equation}
for each $q \in M$. 

Throughout the paper, we will denote by $\mathbb{S}^n$ the Riemannian $n$-dimensional sphere equipped with the round metric and by $\T^n := \R^n/\Z^n$ the $n$-dimensional torus with the flat metric. We say that a manifold is \emph{Anosov} if its geodesic flow is Anosov on its unit tangent bundle, see \S\ref{subsec:resgen} and \S\ref{anosov overview} for further details. In particular, this includes all negatively-curved manifolds.

We will prove the following result.

\begin{thm}
\label{thm:main}
Assume that $M = \mS^n, {\T^n}$ or is Anosov. Then the following holds.
\begin{enumerate}[label=\emph{(\roman*)}]
\item \label{point:imain} There is $c(n, \alpha) > 0$ such that the average of $u_{\varepsilon}$ has the expansion
\[
\frac{1}{|M|}\int_M u_{\varepsilon} {\rm dvol}_g\sim \frac{|M|c(n, \alpha)}{\varepsilon^{n-2\alpha}}, \quad \varepsilon \to 0. \]
\item For each $p \neq p_0 \in M$ (and $p \neq -p_0$ if $M = \mS^n$),
\[
u_{\varepsilon}(p) - \frac{1}{|M|}\int_M u_{\varepsilon}{\rm dvol}_g \to |M| \mathrm{G}_\mathscr{A}(p,p_0), \quad \varepsilon \to 0,
\]
where $\mathrm{G}_\mathscr{A}$ is the Green's function of the generator of $(X_t)_{t\geqslant 0}$ (see Theorem \ref{anosov generator} and Corollary \ref{sphere cor}).\\
\item If $M = \mS^n$, $n > 1 + 4\alpha$ and $1 > (n-4)\alpha$  then for some $\tilde c(n,\alpha)\neq 0$,
$$
\left|u_{\varepsilon}(-p_0) - \frac{1}{M}\int_M u_{\varepsilon}{\rm dvol}_g\right|\sim \frac{|M|\tilde c(n, \alpha)}{\varepsilon^{n -1 -4\alpha }}, \quad \varepsilon \to 0.
$$
\end{enumerate}
\end{thm}

We will state below more precise results (Theorems \ref{narrow capture} and \ref{narrow capture sphere}) giving an explicit expression of the constants and the size of the remainders. While such results exist for Brownian motions in Euclidean domains \cite{singer2008narrow, singer2006narrow, gomez2015asymptotic, cheviakov2010asymptotic, chen2011asymptotic, ammari2012layer} and on general manifolds \cite{nursultanov2022narrow, nursultanov2021mean, nursultanov2021narrow}, this is the first such detailed analytical calculation for L\'evy flights for such a broad class of geometries. 

We emphasize that Theorem \ref{thm:main} shows that the asymptotics of the deviation of the expected stopping time from its average heavily depends on the underlying geometry. In particular on the sphere, antipodal points are conjugate\footnote{On the sphere, conjugate points correspond to pair of points that may be connected by a non-trivial continuous one-parameter family of geodesic paths.}, and this leads to a singular behavior of the expected stopping time at those points. Nevertheless, we expect that point \emph{\ref{point:imain}} of Theorem \ref{thm:main} should remain valid for general Riemannian manifolds, regardless of the geometry.

The work of \cite{getoor1961first} showed that on $\R^n$ the expected time for a L\'evy process to exit the unit ball satisfies an integral equation involving the fractional Laplacian and we will derive an analogous statement in our setting, see Proposition \ref{proposition uepsilon} below. Theorem \ref{thm:main} follows from a detailed study of the analytic properties of the generator $\mathscr{A}$ of the L\'evy process, see Theorems \ref{anosov generator} and \ref{sphere} below. 

We finally observe that in the physical dimensions $n=2$, the expected stopping time for the Brownian motion was shown to be of size $\mc{O}(|\log \varepsilon|)$ in \cite{nursultanov2022narrow} whereas it is here of size $\mc{O}(\varepsilon^{-(2-2\alpha)})$ by Theorem \ref{thm:main}. 

\subsection{Results on the generator}\label{subsec:resgen}

While it is well understood that the infinitesimal generator for $2\alpha$-stable jump processes on Euclidean spaces are precisely the fractional powers of the Laplacian, the same may not hold for L\'evy processes on closed compact Riemannian manifolds. In fact it was shown in \cite{AppEst} that if $(X_t)_{t\geqslant 0}$ is a cadlag semi-martingale valued in a Riemannian manifold $(M,g)$, then it is an isotropic L\'evy process iff it is a Feller process with infinitesimal generator $a\Delta_g +\mathscr{A}$ for some constant $a\geqslant 0$ and for $u\in C^\infty(M)$,
\begin{equation} 
\label{def of A}
{\mathscr{A} }u(p):= \mathrm{p.v.}\int_{v\in T_pM\setminus 0} \left( u(\exp_p(v)) - u(p)\right)\nu_p(\dd v).
\end{equation}
Here $\mathrm{p.v.}$ means that we take the principal value of the integral,
$\{\nu_p\}_{p\in M}$ is a field of measures on $T_pM$ induced from an isotropic L\'evy measure $\nu$ on $\R^n$ by $\nu_p(A) = \nu(r^{-1}(A))$ whenever $\pi(r) = p$ and $r\in {\mathcal O}(M)$ is an element of the orthonormal frame bundle over $M$. Alternatively, one can re-write the principal value of the integral \eqref{def of A} as
\[
\frac{1}{2} \int_{v\in T_pM\setminus 0} \left( u(\exp_p(v)) + u(\exp_p(-v)) - 2u(p)\right)\nu_p(\dd v).
\]
Note that thanks to the isotropic assumption on $\nu$, this definition is independent of the choice of $r\in {\mathcal O}(M)$. 

When the leading term in the generator is $a\Delta_g$ (i.e. $a>0$), some mapping properties were analyzed in \cite{applebaum20212}. However, not much is known about the case when $a=0$ (i.e. the process is "pure jump"). This is due to the fact that \eqref{def of A} is now the dominant driver of the process and integrating the exponential map is difficult to control beyond the injectivity radius on a general Riemannian manifold. We address this challenge for a broad class of Riemannian manifolds.

Throughout the article, we make the choice
\begin{eqnarray}
\label{choice of levy measure}
\nu (A) = C(n,\alpha) \int_{A} \frac{1}{|v|^{n+2\alpha}}\dd v,\quad \alpha\in (0,1),
\end{eqnarray}
for the L\'evy measure, which is motivated by the fact that such processes on $\R^n$ are generated by the fractional Laplacian on the Euclidean space. Note that after pulling back by an element of the fiber of the orthonormal frame bundle $\mathcal O(M)$ over $p\in M$, this measure becomes the L\'evy measure on $T_pM$ described earlier in \eqref{eq:levymeasure}.

We will prove:
\begin{thm}
Let $(X_t)_{t\geqslant 0}$ be a cadlag semi-martingale valued on a Riemannian manifold $(M,g)$ which is either {$\mathbb{S}^n,\T^n$ or Anosov}. If it is an isotropic L\'evy process with pure jump induced by the L\'evy measure \eqref{eq:levymeasure}, then its infinitesimal generator $\mathscr{A}$ is a non-positive Fredholm operator
\[
\mathscr{A} : W^{s,m}(M) \to W^{s-2\alpha, m}(M),
\]
for all $s \in \R, m \in (1,\infty)$, that has discrete spectrum with one dimensional null-space and co-kernel.
\end{thm}

We now give more details on our results on the generator on this L\'evy process.

\subsubsection{Dirichlet form of the generator}

The explicit presence of the exponential map in \eqref{def of A} suggests that the behaviour of $\mathscr{A}$ depends more on the geometry and the dynamics of geodesics than the fractional Laplacian. Therefore, we will take a dynamical systems approach. As an example of the advantages of taking this point of view, we can quickly see that $\mathscr{A}$ is always formally given by a Dirichlet form, as follows.

\begin{prop}
\label{uniqueness}
Let $(M,g)$ be a closed connected Riemannian manifold. There exists an operator ${\mathscr{D}}: {\rm Lip}(M)\to L^2(\R\times SM)$, with kernel given by the constant functions $\mathrm{Ker} ({\mathscr{D}}) = \C \cdot \mathbf{1}$, such that for all $u,v\in C^\infty(M)$, 
\[
-4\int_M u \, \mathscr{A} v\,{\rm dvol}_g = \int_\R \int_{SM} {\mathscr{D}} u \, {\mathscr{D}} v\, \dd \mathrm{L} \dd t.
\]
Here $SM$ denotes the unit sphere bundle and $\mathrm{L}$ is the Liouville measure on $SM$ invariant under the geodesic flow. Consequently, $\mathscr{A}$ is a non-positive operator that admits an extension as an operator ${\rm Lip}(M) \to \D'(M)$ whose kernel consists of constant functions.
\end{prop}
We make the following remarks regarding the previous result:

\begin{enumerate}[label={(\roman*)}]
\item When the transition probability can be obtained by solving the heat equation with infinitesimal generator $\mathscr{A}$, Proposition \ref{uniqueness} implies that there is only one differentiable equilibrium state;
\item When $(M,g)$ is the round sphere, the flat torus, or Anosov, we can drop the differentiability a-priori assumption in Proposition \ref{uniqueness}.
\end{enumerate}


It is natural to ask how similar/different is $\mathscr{A}$ to the fractional Laplacian (defined spectrally) when $(M,g)$ is not Euclidean. We address this question for various manifolds below.

\subsubsection{Generator on the torus}
If $M = \T^n = \R^n / \Z^n$ is the flat torus, the operator $\mathscr{A}$ happens (not surprisingly) to be the fractional Laplacian:

\begin{thm}
\label{torus}
If $(M,g)$ is the torus $\T^n$, the infinitesimal generator given by \eqref{def of A} is
$$
-\mathscr{A} = (-\Delta)^\alpha,
$$ where $\Delta$ is the non-positive Laplace operator on $\T^n$. In particular, $\mathscr{A}$ is an elliptic, classical, pseudo-differential operator of order $2 \alpha$.
\end{thm}

This result is byproduct of \cite[Example 1]{applebaum2014probability} but can also be obtained by a simple explicit computation, which we provide in \S\ref{subsec:generatortorus}. Obviously, for general Riemannian manifolds, such an explicit computation will not be available.

\subsubsection{Generator on the sphere}

It turns out that in the case of the round unit sphere, $\mathscr{A}$ {\bf\it does not} in fact resemble the fractional powers of the Laplacian. It is rather an object belonging to a more general class of operators called Fourier Integral Operators introduced by H\"ormander \cite{Hor}. For a background on microlocal analysis and pseudodifferential operators, we refer the reader to \S\ref{ssection:microlocal}. The consequences for the mapping properties of $\mathscr{A}$ on functional spaces are described in \S\ref{sssection:mapping} below. On the unit sphere with round metric, we denote by $\mathscr{J} : \D'(M)\to \D'(M)$ the pullback by the antipodal map. We will prove that the following result holds.

\begin{thm}
\label{sphere}
If $(M,g)$ is the sphere $\mathbb{S}^n$, the infinitesimal generator given by \eqref{def of A} can be written
{\begin{eqnarray}
\nonumber
\mathscr{A}  = \mathscr{A}_{2\alpha} + \mathscr{A}_0 + \mathscr{A}_{-1}\mathscr{J}
\end{eqnarray}}where for each $\ell = 2\alpha, 0, -1$, $\mathscr{A}_{\ell}\in \Psi^{\ell}_{\mathrm{cl}}(M)$ is a classical formally selfadjoint pseudodifferential operator of  order $\ell$. The operators $\mathscr{A}_{2\alpha}$ and $\mathscr{A}_{-1}$ have principal symbols $\sigma_{\mathscr{A}_{2\alpha}}(x,\eta) = - |\eta|_g^{2\alpha}$ and $\sigma_{\mathscr{A}_{-1}}(x,\eta) = c(n) |\eta|_g^{-1}$, for some constant $c(n) > 0$. All operators commute with the operator $\mathscr{J}$.
\end{thm}

We shall see that since the integral kernel of $\mathscr{A}$ has singularities at both $p=q$ and $p=-q$ (antipodal point), it cannot be the fractional Laplacian.

\subsubsection{Generator on Anosov manifolds}

It is natural to deduce that the complications arising on the sphere are due to geodesics focusing at a single point (i.e. conjugate points). If we make assumptions about the manifold $(M,g)$ as to rule out such behaviour, we should expect $\mathscr{A}$ to have a simpler expression. This is indeed the case {if we assume that $(M,g)$ is \emph{Anosov}. The class of Anosov Riemannian manifolds is a very large class\footnote{In fact, $(M,g)$ is Anosov if and only if $(M,g)$ lies in the $C^2$ interior of the set of metrics without conjugate points \cite{ruggiero1991creation}.}  including in particular all negatively-curved manifolds, see \S\ref{anosov overview} or \cite{anosov1969geodesic, Knieper-02} for further details. We will prove the following result.}

\begin{thm}
\label{anosov generator}
If $(M,g)$ is a closed connected Anosov Riemannian manifold, the infinitesimal generator given by \eqref{def of A} can be written
$$
\mathscr{A} = \mathscr{A}_{2\alpha} + \mathscr{A}_0
$$
where for each $\ell = 2 \alpha, 0$, $\mathscr{A}_\ell \in \Psi^\ell_{\mathrm{cl}}(M)$ is a classical formally selfadjoint pseudodifferential operator of  order $\ell$.

More precisely, for each $\chi \in C^\infty_c(\R, [0, 1])$ such that $\chi(t) = 1$ for $t$ near $0$ and $\mathrm{supp}(\chi) \subset [0, r_{\mathrm{inj}}^2/2]$, where $r_{\mathrm{inj}}$ is the injectivity radius of $(M,g)$, the operator \eqref{def of A} writes
\begin{equation}
\label{anosov kernel}
\begin{aligned}
{\mathscr{A}} u (p) &= C(n,\alpha)~ \mathrm{p.v.}  \int_M \chi(\mathrm{dist}_g(p,q)^2) \frac{u(q)- u(p)}{\mathrm{dist}_g(p,q)^{n+2\alpha}}J(p,q){\rm dvol}_g(q) \\ &+ w(p)u(p)+ \int_M K(p,q) u(q) {\rm dvol}_g(q)
\end{aligned}
\end{equation}
for some smooth functions $w\in C^\infty(M)$ and $K\in C^\infty(M\times M)$. Here we set $J(p,q) = \det \dd_q\exp_{p}^{-1}$. 
\end{thm}


An immediate observation is that when $(M,g)$ is Anosov, the result of Theorem \ref{anosov generator} implies that the operator $\mathscr{A}$ is an elliptic pseudodifferential operator with principal symbol $\sigma_{\mathscr{A}}(x,\xi) = -|\eta|_g^{2\alpha}$ if $\alpha \geqslant 1/2$ and $\sigma_{\mathscr{A}}(x,\xi) = -|\eta|_g^{2\alpha} + \sigma_{\mathscr{A}_0}(x,\xi)$ if $\alpha < 1/2$. Also remark that, when $(M,g)$ is Anosov, the trace formula of Duistermaat-Guillemin \cite{Duistermaat-Guillemin} implies that the spectrum of $\mathscr{A}$ determines uniquely the lengths of periodic geodesics.

\subsubsection{Mapping properties of $\mathscr{A}$}

\label{sssection:mapping}

We will see that Theorems \ref{torus}, \ref{sphere} and \ref{anosov generator} imply the following
\begin{cor}
\label{sphere cor}
If $(M,g)$ is $\mS^n, \mathbb T^n$ or Anosov, the following holds.
\begin{enumerate}[label=\emph{(\roman*)}]
\item\label{item:1cor} $-\mathscr{A}$ extends to a formally selfadjoint Fredholm operator
\[
-\mathscr{A} : W^{s,m}(M) \to W^{s-2\alpha,m}(M),
\]
for all $s \in \R, m \in (1,\infty)$, with non-negative discrete spectrum and smooth eigenfunctions for all $s\in \R$. The null-space consists of only constant functions. \\
\item\label{item:2cor} There exists $\mathscr{A}^+ :W^{s,m}(M) \to W^{s+2\alpha,m}(M)$ with $\mathrm{Ker}(\mathscr{A}^+) = \C \cdot \mathbf{1}$ and $\mathrm{Ran}(\mathscr{A}^+) \perp \C \cdot \mathbf{1}$ such that 
$$\mathscr{A}^+ \mathscr{A} = \mathscr{A} \mathscr{A}^+ = I - P$$
where $P$ is the $L^2$ orthonormal projection to the space of constant functions. The Schwartz kernel of $\mathscr{A}^+$, $ \mathrm{G}_{\mathscr{A}}(p,q)$, which we will call the Green's function, satisfies, for each $p \in M$ and $u \in C^\infty(M)$,
\[
\int_{M} \mathrm{G}_{\mathscr{A}}(p,\cdot) \mathscr{A} u \, {\rm dvol}_g = u(p) - |M|^{-1} \int_M u {\rm dvol}_g
\]
\end{enumerate}
\end{cor}

\subsubsection{Domain of the generator} 
We will see later that in the sphere, torus and Anosov case, the "heat kernel" $e^{t\mathscr{A}}$ has bounded integral kernel (Lemma \ref{Linfty estimate for density}) for all $t>0$ and is therefore trace class. Since the spectrum is discrete and the operator is semidefinite, the solutions of the heat equation converge exponentially in $L^2(M)$ to the constant function because $\mathrm{Ker}(\mathscr{A}) = \C \cdot \mathbf{1}$. At last, we have the Poincar\'e inequality on on the sphere, torus, and Anosov case:
\begin{eqnarray}
\label{poincare}
-\int_M u \, \mathscr{A} u\, {\rm dvol}_g \geqslant c \|u\|^2_{L^2}.
\end{eqnarray}
for all $u\perp\C \cdot \mathbf{1}$.

Our detailed knowledge of $\mathscr{A}$ will yield some insight into the probabilitic aspects of the process $(X_t)_{t\geqslant 0}$.
For each $m\in (1,\infty)$, we define the domain $D_{L^m}(\mathscr{A}) \subset L^m(M)$ to be the set
$$\left \{ u\in L^m(M)~\Big|~ \lim\limits_{t\to 0^+} \left(\mathbb E(u(X_t) \mid X_0 = \cdot) - u\right)/t\ {\text{exists in }} L^m(M)\right\}.$$
We have a precise description of $D_{L^m}(\mathscr{A})$:
\begin{prop}
When $(M,g)$ is $\mS^n, \T^n$ or Anosov, then $D_{L^m}(\mathscr{A})  = W^{2\alpha, m}(M)$ when $m\in (1,\infty)$.
\end{prop}

This proposition actually comes as an intermediate step (Lemma \ref{domain is fractional sobolev}) in the calculation of the expected time it takes to find a small target using a random search which we now describe in detail. 

\subsection{Applications to random searches}\label{subsec:randomsearches}

As in \S\ref{subsec:results}, let $(X_t)_{t\geq0}$ be a cadlag semi-martingale that is an isotropic L\'evy process with infinitesimal generator $\mathscr{A}$ defined by \eqref{def of A}. {Let $B_\varepsilon(p_0)$ be the open geodesic ball of radius $\varepsilon>0$ centred at $p_0$.} We define $\tau_\varepsilon$ and $u_\varepsilon$ by \eqref{eq:deftaueps}. Let $c(n,\alpha)$ be the constant given by
\begin{equation}
\label{equation:cna}
c(n,\alpha) := \begin{cases}
			\dfrac{2^{-2\alpha}(1-\alpha) \Gamma(1-\alpha)^2 }{\pi^2}, & \text{if $n=2$,}\vspace{0.3cm}\\
            \dfrac{2^{1-2\alpha}\Gamma(n/2-\alpha)\Gamma(n/2-\alpha+1)}{\pi^{n/2}(n-2)\Gamma(n/2-1)}, & \text{if $n \geqslant 3$.}
		 \end{cases}
\end{equation}
Then we have the following result, which is a more precise version of Theorem \ref{thm:main}, involving remainder terms.

\begin{thm}
\label{narrow capture}
If $(M,g)$ is a closed connected Anosov Riemannian manifold, then:
\begin{enumerate}[label=\emph{(\roman*)}]
\item As $\varepsilon \to 0$, the average of $u_\varepsilon$ over $M$ has expansion {
\[
\dfrac{1}{|M|} \int_M u_\varepsilon {\rm dvol}_g =\varepsilon^{2\alpha - n} |M| c(n,\alpha) (1+\mc{O}(E(\alpha,\varepsilon))),
\]
}
where the error term $E(\alpha, \varepsilon)$ is given by
\begin{eqnarray}
\label{eq:errorterm}
E(\alpha, \varepsilon) = 
\begin{cases}
\varepsilon^{2\alpha}, \quad &\text{if} \ \alpha < 1/2, \\
  \varepsilon|\log\varepsilon|, \quad &\text{if} \ \alpha = 1/2,\\
\max( \varepsilon, \varepsilon^{n-2\alpha}), \quad &\text{if} \ \alpha > 1/2.
\end{cases}
\end{eqnarray}

\item For all $\varepsilon >0$, $u_\varepsilon \in C^\infty(M\setminus \partial B_\varepsilon (p_0)) \cap L^\infty(M)$. Moreover, for all $p\neq p_0$, we have as $\varepsilon \to 0$
\begin{equation}\label{eq:asymueps(p)}
u_\varepsilon(p)  - \dfrac{1}{|M|} \int_M u_\varepsilon {\rm dvol}_g =  |M| \mathrm{G}_{\mathscr{A}}(p,p_0) +\mc{O}(E(\alpha,\varepsilon))
\end{equation}
where $ \mathrm{G}_{\mathscr{A}}(p,q)$ is the Green's function of $\mathscr{A}$ given by (iii) of Theorem \ref{anosov generator}.
\end{enumerate}
\end{thm}

For the torus, the same result holds, up to changing the error term:

\begin{thm}
\label{narrow capture torus}
If $(M,g)$ is $\T^n$, then the conclusions of Theorem \ref{narrow capture} hold if we replace the error term \eqref{eq:errorterm} by
\begin{equation}\label{eq:newerrorterm}
E(\alpha, \varepsilon) = \begin{cases}
\max(\varepsilon,\varepsilon^{n-2\alpha}), \quad &\text{if} \ \alpha \neq 1/2, \\
  \varepsilon|\log\varepsilon|, \quad &\text{if} \ \alpha = 1/2.
\end{cases}
\end{equation}
\end{thm}

These asymptotics are similar to the ones computed in \cite{nursultanov2021mean,nursultanov2022narrow} for the Brownian motion. When $\alpha>0$ is small the situation on the sphere is quite different from that of Anosov manifolds. Due to the singularity structure of $\mathscr{A}$ when $M= \mathbb{S}^n$, a propagation phenomena occurs from $p_0$ to $-p_0$ to create,  as $\varepsilon \to 0$, a blowup of the quantity 
\[
\left|u_\varepsilon(-p_0) - |M|^{-1}\int_M u_\varepsilon{\rm dvol}_g\right|.
\]
We will prove that the following holds:

\begin{thm}
\label{narrow capture sphere}
If $(M,g)$ is $\mathbb{S}^n$, then:

\begin{enumerate}[label=\emph{(\roman*)}]
\item The average value of $u_\varepsilon$ over $M$ 
is the same as in Theorem \ref{narrow capture}.
\item For all $\varepsilon>0$, we have
$$u_\varepsilon \in C^\infty\left(M\setminus\left(\partial B_\varepsilon(p_0)\cup \partial B_\varepsilon (-p_0)\right)\right)\cap L^\infty(M)$$ and \eqref{eq:asymueps(p)} holds whenever $p\notin\{p_0, -p_0\}$ where $ \mathrm{G}_{\mathscr{A}}$ is given by Corollary \ref{sphere cor}.

\item If $n>1+4\alpha$ and $1>(n-4)\alpha$, then at $p= -p_0$ we have
{\begin{eqnarray}
\label{sphere blow up}
\left|u_\varepsilon(-p_0) - \dfrac{1}{|M|}\int_Mu_\varepsilon{\rm dvol}_g\right|=\frac{\tilde c(\alpha, n)|M|}{\varepsilon^{n-1-4\alpha} }
+ o(\varepsilon^{-n+1+4\alpha})
\end{eqnarray}}
for some $\tilde c(\alpha,n)>0$, {which we do not make explicit}.
\end{enumerate}

\end{thm}

Following Theorem \ref{narrow capture sphere}, it would interesting to understand the generator $\mathscr{A}$ and the narrow capture problem in other settings than the sphere where conjugate points appear, like Zoll manifolds for instance. This is left for future investigation.




\subsection{Structure of the paper}
In \S\ref{sec:preli}, we recall some general facts of microlocal analysis and Riemannian geometry. In \S\ref{sec:propgen}, we study the generator $\mathscr{A}$ and prove the results announced in \S\ref{subsec:resgen}. In \S\ref{sec:stoppingtime}, we prove the results announced in \S\ref{subsec:randomsearches} on the expected stopping time, omitting technical results on the solutions of the integral equation $\mathscr{A} u_\varepsilon = -1$ on $M \setminus B_\varepsilon(p_0)$, which we leave until \S\ref{sec:technical}.

\subsection*{Acknowledgements} The authors wish to thanks David Applebaum, Sonja Cox, and Frank Redig for the useful discussions and encouragement during the writing of this article.

\section{Preliminaries}

\label{sec:preli}

In this section, we detail some tools needed throughout the paper.

\subsection{Microlocal analysis}

\label{ssection:microlocal}

We refer to \cite{Grigis-Sjostrand-94, hormander} for a general treatment.

\subsubsection{Pseudodifferential operators}

\label{sssection:pdo}

Let $M$ be a closed $n$-dimensional manifold. For $k \in \R$, we define $S^k(T^*M) \subset C^\infty(T^*M)$, the space of symbols of order $k$, as the set of smooth functions $a$ satisfying the following bounds, in any coordinate chart $U \subset \R^n$: for all $\gamma,\beta \in \N^n$, there exists $C := C(U,\alpha,\beta) > 0$ such that
\begin{equation}
\label{equation:symbol}
\forall (x,\xi) \in T^* U \simeq \R^n \times \R^n, \qquad |\partial^\gamma_\xi \partial^\beta_x a(x,\xi)| \leqslant C \langle \xi \rangle^{k-|\gamma|}.
\end{equation}
It can be checked that \eqref{equation:symbol} is invariant by diffeomorphism, which implies that $S^k(T^*M)$ is intrinsically defined on $M$.

We define $\Psi^{-\infty}(M)$, the set of \emph{smoothing operators}, as the space of linear operators on $M$ with smooth Schwartz kernel. Denote by $\mathrm{Op}$ a quantization procedure on $M$, given in a local coordinate patch $U \subset \R^n$ by: 
\[
\mathrm{Op}(a)f (x) = \dfrac{1}{(2\pi)^n} \int_{\R^n_\xi} \int_{\R^n_y} e^{i\xi\cdot(x-y)} a(x,\xi) f(y) \dd y \dd\xi,
\]
where $a \in S^k(T^*U)$ and $f \in C^\infty_{c}(U)$. The set of \emph{pseudodifferential operators} of order $k \in \R$ is then defined as
\[
\Psi^k(M) := \left\{ \mathrm{Op}(a) + R ~|~ a \in S^k(T^*M), R \in \Psi^{-\infty}(M)\right\}.
\]
It can be checked that $\Psi^k(M)$ is intrinsically defined and independent on the choice of quantization $\mathrm{Op}$.

There exists a well-defined \emph{principal symbol map}
\[
\sigma : \Psi^k(M) \to S^k(T^*M)/S^{k-1}(T^*M)
\]
such that we have the following exact sequence:
\[
0 \longrightarrow \Psi^{k-1}(M) \longrightarrow \Psi^k(M) \longrightarrow S^k(T^*M)/S^{k-1}(T^*M) \longrightarrow 0.
\]
The elliptic set $\mathrm{ell}(A) \subset T^*M \setminus \left\{0\right\}$ of an operator $A \in \Psi^k(T^*M)$ is defined as the (open) conic set of points $(x_0,\xi_0) \in T^*M \setminus \left\{ 0\right\}$ such that there exists a constant $C > 0$ such that the following holds:
\begin{equation}
\label{equation:elliptic}
\left(|\xi| \geqslant C \text{ and } d_{S^*M}\left((x,\xi/|\xi|), (x_0,\xi_0/|\xi_0|) \right) < 1/C \right) \implies |\sigma_A(x,\xi)| \geqslant \langle \xi\rangle^k/C.
\end{equation}
Here $d_{S^*M}$ is any metric on the cosphere bundle $S^*M := T^*M/\R_+$, where the $\R_+$-action is given by radial dilation in the fibers of $T^*M$. An operator is said to be \emph{elliptic} if $\mathrm{ell}(A) = T^*M \setminus \left\{0\right\}$. The \emph{characteristic set} $\Sigma(A)$ of an operator is the closed conic subset defined as the complement of the elliptic set in $T^*M$. The important property of elliptic operators (on $T^*M$) is that they are invertible modulo smoothing remainders, that is, one can find $B \in \Psi^{-k}(T^*M)$ and $R \in \Psi^{-\infty}(M)$ such that
\[
B A = \mathbf{1} +R.
\]
Such an operator $B$ is called a \emph{parametrix} for $A$.

\subsubsection{Wavefront set of distributions}

The \emph{wavefront set} $\WF(A)$ (or the \emph{microsupport}) of an operator $A \in \Psi^{k}(M)$ is the (closed) conic subset of $T^*M \setminus \left\{0\right\}$ satisfying the following property: $(x_0,\xi_0) \notin \WF(A)$ if and only if for all $m \in \R$, for all $b \in S^m(T^*M)$ supported in a small conic neighborhood of $(x_0,\xi_0)$, one has $A\mathrm{Op}(b) \in \Psi^{-\infty}(M)$. In other words, the complement of the wavefront set of $A$ is the set of codirections where $A$ behaves as a smoothing operator.

The wavefront set $\WF(u)$ of a distribution $u \in \mathcal{D}'(M)$ is the (closed) conic subset of $T^*M \setminus \left\{0\right\}$ satisfying the following property: $(x_0,\xi_0) \notin \WF(u)$ if and only if there exists a small open conic neighborhood $V$ of $(x_0,\xi_0)$ such that for all $k \in \R$, for all $A \in \Psi^k(M)$ with wavefront set contained in $V$, one has $A u \in C^\infty(M)$. In particular, a distribution/function $u$ is smooth if and only if $\WF(u) = \emptyset$. Equivalently, the wavefront set of a distribution can be characterized as follows: taking $(x_0,\xi_0) \in T^*M \setminus \left\{0\right\}$, the point $(x_0,\xi_0)$ is not in the wavefront set of $u$ if we can find $\chi, S \in C^\infty(M)$ such that $\chi$ has support near $x_0$, $dS \neq 0$ on the support of $\chi$, $dS(x_0) = \xi_0$, and
\begin{equation}
\label{equation:decay}
\langle u, e^{-i S/h}\chi\rangle = \mathcal{O}(h^\infty).
\end{equation}

\subsubsection{Functional spaces}

\label{sssection:functional-spaces}

We now introduce the functional spaces we will be working with. We denote by $\Delta_g \leqslant 0$ the negative Hodge Laplacian acting on functions. For all $s \in \R$, the operator $(\mathbf{1}-\Delta)^s$ defined using the spectral theorem (applied to the selfadjoint operator $\Delta_g$ on $L^2(M, \mathrm{vol}_g)$) is an invertible pseudodifferential operator of order $2s$.

For $s \in \R$, $m \in (1,\infty)$ and $u \in C^\infty(M)$, we set
\begin{equation}
\label{equation:norm}
\|u\|_{W^{s,m}} := \|(\mathbf{1}-\Delta)^{s/2}u\|_{L^m},
\end{equation}
and define $W^{s,m}(M)$ to be the completion of $C^\infty(M)$ with respect to the norm \eqref{equation:norm}. Taking $m=2$, we retrieve the usual Sobolev spaces which we will rather denote by $H^s(M) := W^{s,2}(M)$. Note that the spaces $W^{s,m}(M)$ intrinsically defined, that is, they are independent of the choice of metric $g$, and changing the metric only replaces the norm \eqref{equation:norm} by an equivalent norm. 

The following boundedness result for pseudodifferential operators holds: for all $k \in \R, A \in \Psi^{k}(M)$ and $s \in \R$, $m \in (1,\infty)$,
\begin{equation}
\label{equation:boundedness}
A : W^{s+k, m}(M) \to W^{s,m}(M)
\end{equation}
is bounded.

Eventually, given $\Omega\subset M$ be an open subset with non-empty smooth boundary, we define, for $s\in \R$ and $m \in [1,\infty)$, the spaces
$$ \dot W^{s, m} (\overline \Omega) := \{u\in W^{s,m}(M) \mid \supp(u) \subset \overline \Omega\}$$
and
$$\overline W^{s,m}(\Omega) := \{u|_{\Omega} ~|~ u\in W^{s,m}(M) \}.$$

\subsection{Riemannian geometry of the unit sphere bundle}

Let $(M,g)$ be a smooth closed connected $n$-dimensional Riemannian manifold. The unit sphere bundle $SM$ over $M$ is defined by
\begin{eqnarray}
\label{sphere bundle}
SM:= \{(p,v)\in TM\mid |v|_g = 1\}.
\end{eqnarray}
Since $M$ has dimension $n$, $SM$ is a manifold of dimension $2n-1$. A generic point in $SM$ will be denoted by $z$.

Associated to this operator is the pull-back operator $\pi^* : \mathcal D'(M)\to \mathcal D'(SM)$, which, when restricted to $u\in C^\infty(M)$ takes the form $(\pi^* u)(p,v) := u(p)$. Note that $\pi^* : C^\infty(M)\to C^\infty(SM)$. It has an adjoint, the push-forward $\pi_* : \mathcal D'(SM) \to \mathcal D(M)$ defined via the adjoint relation 
$$\langle \pi_* u ,\varphi\rangle := \langle u, \pi^*\varphi\rangle.$$
When $u\in C^\infty(SM)$ the push-forward can be written explicitly as
$$(\pi_*u)(p) = \int\limits_{v\in S_pM} u(p,v) \dd S^{n-1}(v).$$

Let $X_g$ be the geodesic vector field on $SM$ and $\phi_t(\cdot)$, $t\in \R$ be the geodesic flow generated by this vector field. Using the canonical projection $\pi : SM \to M$ which maps $\pi: (p,v)\mapsto p$, we define the vertical bundle $\bV \subset TSM$ by:
\begin{equation}
\label{vertical bundle}
\bV = \{ (p,v, V)\in TSM ~|~ \dd\pi_{(p,v)} V = 0 \}.
\end{equation}
The tangent space to $SM$ then splits as
\begin{equation}
\label{split}
TSM = \R X_g \oplus \bV \oplus \bH,
\end{equation}
where $\bH$ is the horizontal bundle, which can be defined as the horizontal space of the Levi-Civita connection induced by $g$, see \cite{Paternain-99} for instance. The metric $g$ induced a natural metric on $SM$, called the Sasaki metric and denoted by $G$, for which the splitting \eqref{split} is orthogonal.

The Riemannian measure induced by the Sasaki metric $G$ is called the Liouville measure $\dd\mathrm{L}$. It is invariant under the flow, that is,
\begin{eqnarray}
\label{flow invariance}
\int_{SM} u(\phi_t(z)) w(z) \dd\mathrm{L}(z) = \int_{SM} u(z) w(\phi_{-t}(z)) \dd\mathrm{L}(z) 
\end{eqnarray}
for all $u, w \in C(SM)$ and $t\in \R$.
There is a convenient way to describe the measure $\dd\mathrm{L}$ locally as the product of the measure on the round unit sphere $S^{n-1}$ and the Riemannian volume on $M$:
\begin{eqnarray}
\label{split the measure}
\dd \mathrm{L}(p,v)  = \dd S^{n-1}(v) \wedge {\rm dvol}_g(p).
\end{eqnarray}

\subsection{Overview of Anosov manifolds}

\label{anosov overview}

We briefly recall the definition and some basic properties of Anosov manifolds. Let $(M,g)$ be a closed compact manifold with geodesic vector field $X_g$ whose flow $\phi_t(\cdot)$ is complete. We say that the manifold $(M,g)$ is Anosov if its geodesic flow $\phi_t(\cdot)$ is Anosov, that is, if there is a continuous splitting of invariant bundles ${E_u}$ and ${E_s}$:
\begin{eqnarray}
\label{anosov splitting}
TSM = {E_s} \oplus {E_u} \oplus \R X_g
\end{eqnarray}
and there exists $C>0$ and $0<\rho<1<\eta$ such that for all $t>0$
\begin{equation}
\label{equation:anosov}
\|\dd\phi_{-t}|_{E_u}\|< C\eta^{-t},\ \|\dd\phi_t|_{{E_s}}\|< C\rho^t.
\end{equation}
The norms $\|\bullet\|$ in \eqref{equation:anosov} are computed with respect to an arbitrary auxiliary smooth metric on $SM$ but both properties \eqref{anosov splitting} and \eqref{equation:anosov} are independent of this choice of auxiliary metric. Throughout we will use $p$ to denote points on $M$ and $v\in S_pM$. Sometimes it is convenient to denote $(p,v)$ as a single point in $SM$ in which case we will write $z = (p,v)$.

Associated with the flow $\phi_t$ is the symplectic lift $\Phi_t : T^*SM \to T^*SM$ defined by
\[
\Phi_t(z,\xi) = (\phi_t(z), \dd_z \phi_t^{-\top} \xi).
\]
The dual splitting of \eqref{anosov splitting} is given by invariant vector bundles $E_s^*$ and $E_u^*$
\begin{eqnarray}
\label{dual splitting}
T^*SM = E_s^* \oplus E_u^* \oplus \R \vartheta,
\end{eqnarray}
where $E_s^*=({E_s} \oplus \R  X_g)^\perp$, $E_u^*=({E_u}\oplus \R X_g)^\perp$, while $\vartheta$ is the Liouville 1-form on $SM$. 

The covertical bundle is defined by:
\begin{equation}
\label{covertical bundle}
\bV^\perp = \{( p,v,\xi)\in T^*SM ~|~ \xi(V) = 0,\ \forall V\in \mathbf V\}.
\end{equation}
Assuming that $X_g$ generates an Anosov flow, we have by \cite[Proposition II.2]{mane}:
\begin{equation}
\label{vertical does not intersect stable and unstable}
\bV^\perp \cap E_s^* = \bV^\perp\cap E^*_u = \{0\}.
\end{equation}
Moreover, Anosov manifolds are free of conjugate points \cite{klingenberg1974riemannian} which implies that
\begin{equation}
\label{equation:cp}
d\phi_t(\bV^\perp) \cap \bV^\perp = E_0^*,
\end{equation}
for all $t \neq 0$.

The bundles $E_s \oplus E_0$ and $E_u$ are integrable and tangent to a foliation which consists of \emph{central stable} and \emph{unstable} leaves. We denote the leaves by $W^{s0}(z)$ and $W^u(z)$. They are smooth immersed submanifold in $SM$. They can be defined alternatively by:
\begin{equation}
\label{equation:w}
W^u(z) = \left\{z' \in SM ~|~ \dd_{SM}(\phi_{-t}(z),\phi_{-t}(z')) \to_{t \to +\infty} 0\right\},
\end{equation}
(and similarly for $W^s(z)$ by changing $-t$ to $t$) and
\[
W^{s0}(z) = \bigcup_{t \in \R} \phi_t(W^s(z)).
\]
Note that the convergence in \eqref{equation:w} is actually exponentially fast. We will use the notation $W^{s0}_{\mathrm{loc}}(z)$ (resp. $W^u_{\mathrm{loc}}(z)$) to denote the intersection of $W^{s0}(z)$ (resp. $W^u_{\mathrm{loc}}(z)$) with a small ball $B_{\varepsilon_0}(z)$, where $\varepsilon_0 > 0$ is some small fixed constant.

The Anosov foliation in central stable/unstable leaves is not smooth (it is only transversally continuous to the leaves). Nevertheless, we can introduce the algebra of functions $C^\infty_{\mathrm{cs}}(SM)$ of functions such that their restriction to all central stable leaves are smooth and vary continuously in the unstable direction, see \cite[Section 2]{DeLaLlave-Marco-Moriyon-86}. Similarly, we can construct vector fields which are tangent to every central stable leaves, smooth in restriction to each leaf and vary continuously transversally in the unstable direction. The following lemma connects the wavefront set property of a function/distribution to the smoothness of its restriction to every leaf of this foliation. Its proof is a standard calculation, see \cite[Lemma 1.9]{Bonthonneau-Guillarmou-Weich-20} for instance. 

\begin{lem}
\label{lemma:wf}
Let $f \in C^0(SM)$ be a continuous function. Fix $z_0 \in SM$ and consider a family of vector fields $S_0=X_g,S_1,...,S_{n-1} \in C^\infty_{\mathrm{cs}}(SM,E_s \oplus E_0)$ spanning locally $E_s\oplus E_0$ near $z_0$. If $S^\beta f \in C^0(SM)$ for all multi-indices $\beta \in \N^n$, where $S^\beta := S_0^{\beta_0} S_1^{\beta_1} ... S_{n-1}^{\beta_{n-1}}$, then $\WF(f) \subset E_s^*$.
\end{lem}

Another crucial property of the Anosov foliation is that it is \emph{absolutely continuous}, that is, we can disintegrate smooth measures along stable/unstable leaves and the disintegrated measures are themselves smooth. In other words, if $\dd \mu$ is a smooth measure on $SM$, given $f \in C^\infty(U)$ where $U \subset SM$ is an open subset, we have:
\begin{equation}
\label{equation:desint}
\int_U f(z) \dd \mu(z) = \int_{W^u_{\mathrm{loc}}(p)} \left( \int_{W^{cs}_{\mathrm{loc}}(x)} f(y) \delta_x(y) \dd m^{cs}_x(y)\right) \dd m^{u}_p(x),
\end{equation}
where $p \in U$, $m^{cs}_x$ is the smooth Riemannian measure induced by the restriction of the Sasaki metric $G$ to the leaf $W^{cs}_{\mathrm{loc}}(x)$, $m^u_p$ is the smooth Riemannian measure induced by the Sasaki metric restricted to $W^u_{\mathrm{loc}}(p)$, and $y \mapsto \delta_x(y)$ are smooth in restriction to every leaf, and continuous transversally in the unstable direction, that is, $\delta \in C^\infty_{\mathrm{cs}}(U)$. We refer to \cite[Proposition 1.6]{Bonthonneau-Guillarmou-Weich-20} for a proof.

\subsection{Markov and Feller processes. Capture problem}
\label{sec:prelim-markov}

In order relate the functional analytic properties of the generator $\mathscr{A}$ to the study of the expected capture time $u_\varepsilon$, we will rely on the results of Geetor \cite{getoor1957additive,getoor1959markov,getoor1961first} showing that $u_\varepsilon$ has a convenient integral representation. The aim of this paragraph is to recall the main ingredients of this construction.

For each $t>0$, let ${\mathbf p}(t,\cdot,\cdot) \in \mathcal D'(M\times M)$ be the fundamental solution of the heat equation with generator $\mathscr{A}$ in the sense that
\begin{equation}
\label{transition probability p}
\partial_t \mathbf {\mathbf p}(t, p, \cdot) = \mathscr{A} {\mathbf p}(t,p,\cdot),\ \ {\mathbf p}(0,p, \cdot) = \delta_p(\cdot).
\end{equation}
We let $X := (X_t)_{t \geqslant 0}$ be a cadlag semi-martingale on $M$ which is an isotropic L\'evy process in the sense of \cite{AppEst}, induced by the isostropic L\'evy measure \eqref{eq:levymeasure}.
By construction, $X$ is a Feller process with generator $\mathscr{A}$, see \cite{applebaum20212}. This means that
\[
U_t \phi(p) = \mathbb{E}\left[ \phi(X_t) \ \middle|~ X_0 = p \right],
\]
defines a semi-group satisfying the following properties:
\begin{itemize}
	\item If $f\in C^0(M)$, $\| U_t f \|_{C^0} \leqslant \|f \|_{C^0}$, and $U_t f \to f$ a.e as $t\to 0$;
	\item If $f\in C^0(M)$ satisfies $f\geqslant 0$, then $U_t f \geqslant 0$;
	\item $U_t \mathbf{1} = \mathbf{1}$;
	\item The generator has domain containing $C^\infty(M)$, and coincides with $\mathscr{A}$ on $C^\infty(M)$\footnote{That $C^\infty(M)$ is contained in the domain is not explicitly stated in \cite{AppEst}. For this, \cite{AppEst} relies on \cite[Page 177]{App-95}.};
	\item For any Borel sets $\Omega_1, \Omega_2 \subset M$, the transition probability for the process $(X_t)_{t\geqslant 0}$ is given by
\[
{\mathbb P}(X_t \in \Omega_1 \mid X_0 \in \Omega_2) = |\Omega_2|^{-1}\int_{q \in \Omega_1}\int_{p \in \Omega_2} \mathbf p(t,p,q){\rm dvol}_g(q){\rm dvol}_g(p).
\]
\end{itemize} 

In our case, we will see in Lemma \ref{Linfty estimate for density} that since $\mathscr{A}$ is elliptic pseudo-differential (or almost so in the case of the sphere) the kernel $\mathbf{p}$ is smooth for $t>0$. We will also see that $\mathscr{A}$ is essentially self-adjoint. This implies in particular that $\mathbf{p}(t,p,q)=\mathbf{p}(t,q,p)$, and that $e^{t\mathscr{A}}$ is Hilbert-Schmidt for $t>0$. 
%

We now review the results from \cite{getoor1959markov}. From the aforementioned properties of the kernel $\mathbf{p}$, the assumptions (P) and (K) from \cite{getoor1959markov} are satisfied and the results of \cite[\S 2, 5 and 6]{getoor1959markov} and \cite[Theorem 4.1]{getoor1959markov} apply. For an open set $\Omega \subset M$ and $V: \Omega \to \R^+$ a measurable function, Geetor introduces the operators $T_t= T_t[V,\Omega]$\footnote{The open set $\Omega$ is called $G$ in his notation.} defined by
\begin{equation}\label{eq:def-T_t}
T_t  \phi(p) = \mathbb{E}\left[ \phi(X(t)) e^{-\int_0^t V(X(\tau))d\tau} \mathbf{1}_{\{X(\tau)\in \overline{\Omega},\ 0\leq\tau\leqslant t\}} \middle| X(0)=p\right].
\end{equation}

We have $\|T_t \phi\|_{L^\infty(M)} \leqslant \|\phi\|_{L^\infty(M)}$ for every $\phi\in C^0(\overline{\Omega})$. If $|\partial \Omega| =0$ and for a.e $p \in \Omega$,
\[
\int_0^t \int_{\Omega} p(\tau,p, q) V(y) {\rm dvol}_g(q) d\tau \to_{t \to 0} 0,
\]
then \cite[Theorem 2.1]{getoor1959markov} asserts that $(T_t)_{t \geqslant 0}$ is a strongly continuous semi-group on $L^2(\Omega)$ (in particular, this is certainly true if $V$ is $L^\infty(M)$). Next, according to \cite[Theorem 5.1]{getoor1959markov}, $T_t$ has a measurable kernel $\mathbf{k}(t,p,q)$ satisfying for every fixed $t \geqslant 0,p \in M$
\begin{equation}\label{density comparison}
0 \leqslant \mathbf{k}(t,p,q)\leqslant \mathbf{p}(t,p,q),\ \text{ for a.e }\ q \in M.
\end{equation}

An important feature is the following approximation result \cite[Theorem 4.1]{getoor1959markov}: given an open set $\Omega$, setting $V_\ell := \ell \mathbf{1}_{M \setminus \Omega}$ for $\ell \geqslant 0$, we have that for every $t > 0$ and $\phi\in L^2(\Omega)$,
\begin{equation}
\label{eq:approximating-semi-group}
T_t[V_\ell, M] \phi \to_{\ell \to +\infty} T_t[0,\Omega]\phi, 
\end{equation}
where the convergence holds $L^2(M)$. Furthermore, in \cite[Theorems 6.1, 6.2, 6.3 and 6.4]{getoor1959markov}, it is proved that for $t \geqslant 0$, $T_t$ is self-adjoint, a $L^2(\Omega)$-contraction, Hilbert-Schmidt, positive definite, and if $\mathscr{A}_{V,G}$ is its generator with eigenvalues $0\geqslant \lambda_0 \geqslant \dots \geqslant \lambda_m \geqslant ...$, and eigenfunctions $\{\phi_j\}_{j\geqslant 0}$, then:
\begin{equation}\label{eq:decomp-k}
\mathbf{k}(t,p,q) = \sum_{j\geqslant 0} e^{t \lambda_j} \phi_j(p)\phi_j(q). 
\end{equation}
Finally, let us give a word on the action on $L^m(\Omega)$-spaces for $m \in [1,\infty]$. From the symmetry of $T_t$, that is $\mathbf{k}(t,p,q) = \mathbf{k}(t,q,p)$ for all $p,q \in M$, we deduce that for every $t \geqslant 0,q \in M$,
\begin{equation}
\label{eq:bounded-on-L1}
\int_{\Omega} \mathbf{k}(t,p,q) {\rm dvol}_g(p) = 1.
\end{equation}
This implies that $T_t : L^1(\Omega)\to L^1(\Omega)$ has norm at most $1$. By interpolation, $T_t$ extends as a contraction on every $L^m(\Omega)$, $m \in [1,\infty]$. 

We will use the previous constructions in our arguments in \S\ref{ssection:uniqueness}. More specifically, denoting $\Omega_\varepsilon = M \setminus B_\varepsilon(p_0)$, we set 
\begin{equation}\label{eq:def-semi-groups}
T_t:=T_t[0,\Omega_\varepsilon],\qquad T_t^\ell := T_t[V_\ell,M],
\end{equation}
where $T_t[V_\ell, M]$ are the approximating semi-groups as in \eqref{eq:approximating-semi-group}. It should also be observed that in the proof of \eqref{eq:approximating-semi-group}, one can replace $L^2$ by $L^m$ without any problem. 

For all $p\in \Omega_\varepsilon$, let
\[
\mathcal T_\varepsilon (t,p) := {\mathbb P}(\tau_\varepsilon > t\mid X_0 = p) 
\]
be the probability that a process starting at $p$ does not exit $\Omega_{\varepsilon}$ before time $t>0$. Taking $\Omega=\Omega_\varepsilon$ and $V=0$ in the above definitions, we get
\begin{equation}
\label{P has density}
\mathcal T_\varepsilon(t,p) = {\mathbb P}(X_\tau \in \Omega_{\varepsilon}, \forall 0 \leqslant t \leqslant \tau \mid X_0 = p) = \int_{\Omega_\varepsilon}{\mathbf k}(t, p,q) {\rm dvol}_g(q) = T_t \mathbf{1}_{\Omega_{\varepsilon}},
\end{equation}
and thus
\begin{equation}\label{eq:integral-representation}
u_\varepsilon(p) = \int_0^{+\infty} \int_{\Omega_\varepsilon} \mathbf{k}(t,p,q) {\rm dvol}_g(q) \dd t. 
\end{equation}
To obtain the desired results on $u_\varepsilon$, we will study the finer properties of $\mathbf{k}$ using arguments very similar to \cite{getoor1961first} in \S \ref{ssection:uniqueness}. 

\section{Properties of the generator}\label{sec:propgen}

\subsection{Dirichlet form of the generator}
As a quick demonstration of the advantages of working on the level of dynamical systems, we prove the existence of a Dirichlet form for $\mathscr{A}$ as stated in Proposition \ref{uniqueness}. We will in fact prove the existence of Dirichlet forms for operators which are slightly more general:
\begin{prop}
\label{uniqueness'}
Let $b\in C^\infty(\R)$ satisfy $0\leqslant b(t)\leqslant 1$. Then the operator given by 
$${\mathscr{A}}_b u(p) : =  \int_{v\in T_pM} \left( u(\exp_p(v)) - u(p)\right)b(|v|_g^2) \nu_p(dv)$$
where $\nu_p$ is as in \eqref{eq:levymeasure} satisfies
$$-4\langle u,\mathscr{A}_b v\rangle = \int_0^\infty\int_{SM} {\mathscr{D}}_b u\,  \overline{{\mathscr{D}}_b v} \,\dd \mathrm{L} \dd t$$
for some ${\mathscr{D}}_b: {\rm Lip}(M)\to L^2(\R\times SM)$. When $b(t)$ is strictly positive, ${\mathscr{D}}_b$ is has null-space consisting of only constant functions.
\end{prop}

Set $I_\varepsilon$ to be the indicator function of $\R\setminus (-\varepsilon,\varepsilon)$ and define
$${\mathscr{A}}_{b, \varepsilon} u(p) : =  \int_{v\in T_pM} I_\varepsilon(|v|_{g})\left( u(\exp_p(v)) - u(p)\right)b(|v|_g^2)\nu_p(\dd v).$$
Using the structures on $SM$ introduced above, we have the following representation for ${\mathscr{A}}_\varepsilon$ which is essentially polar coordinates (see \cite{GSUbook}). We will derive it for the convenience of the reader:

\begin{lem} For $u\in C^1(M)$ we have
\label{A in polar}
\begin{eqnarray}\label{eq:lemaeps}
{\mathscr{A}}_{b,\varepsilon} u(p) = \frac{1}{2}\int_{-\infty}^\infty\int_{v\in S_pM}I_\varepsilon(t)\frac{ u(\exp_p(tv)) - u(p)}{|t|^{1 +2\alpha}} b(t^2) \dd S^{n-1}(v)\dd t.
\end{eqnarray}
\end{lem}
\begin{proof}

Fix $p\in M$ and denote by $\tilde {\mathscr{A}}_{b,\varepsilon} u(p)$ the right-hand side of \eqref{eq:lemaeps}.
First observe that we can split the time integration into the sum of integrals on $(-\infty,0)$ and $(0,\infty)$. Applying a change of variable $(t,y)\mapsto (-t,-y)$ to the $(-\infty, 0)$ integral yields

$$\tilde {\mathscr{A}}_{b,\varepsilon} u(p) = \int_{0}^\infty\int_{v\in S_pM}I_\varepsilon(t)\frac{ u(\exp_p(tv)) - u(p)}{|t|^{1 +2\alpha}} b(t^2)\dd S^{n-1}(v)\dd t.$$
The change of variable $\tilde v = tv$ yields (see \cite[Lemma 8.1.8]{GSUbook})

$$\tilde {\mathscr{A}}_\varepsilon u(p) =\int_{ v \in T_pM\setminus 0}I_\varepsilon(| v|_{g})\frac{ u(\exp_p( v)) - u(p)}{| v|_{g}^{n+2\alpha}} b(|v|_g^2)  \dd T_p(\tilde v)$$
where $\dd T_p(v)$ is the volume form on $T_pM$ for the metric $g_p$. 
So we have that $\tilde {\mathscr{A}}_\varepsilon= {\mathscr{A}}_\varepsilon$.
\end{proof}
We now define the operator ${\mathscr{D}}_b: {\rm Lip}(M)\to L^2(\R\times SM)$ 
by
\begin{equation}
\label{def of R}
{\mathscr{D}}_bu (t,p,v):=\sqrt{ b(t^2)} |t|^{-\frac{1+2\alpha}{2}}\left( u(\exp_p(tv) - u(p)\right), \quad t \neq 0, \quad (p, v) \in SM.
\end{equation}
Observe that if $b(t)$ is strictly positive,
\begin{equation}
\label{kernel of D is trivial}
\mathrm{Ker}({\mathscr{D}}_b) = \C \cdot \mathbf{1}
\end{equation}since any two points on $M$ can be joined by a geodesic.
We are now in a position to prove Proposition \ref{uniqueness}:
\begin{proof}[Proof of Proposition \ref{uniqueness'}]
Without loss of generality we may assume that $u,w\in C^\infty(M)$ re both supported in a single coordinate patch. We first define
$$N_\varepsilon := \int_{-\infty}^\infty\int_{SM}  I_\varepsilon(t)\,{\mathscr{D}}_b u\, \overline{{\mathscr{D}}_b w}\, \dd\mathrm{L}\dd t$$
and write
\begin{eqnarray} 
\label{L limit}
\int_{-\infty}^\infty\int_{SM}{\mathscr{D}}_b u \overline{{\mathscr{D}}_b w} \dd\mathrm{L}\dd t = \lim\limits_{\varepsilon \to 0} N_\varepsilon
\end{eqnarray}
Now 
\begin{eqnarray*}\nonumber
N_\varepsilon &=& \int_{-\infty}^\infty b(t^2)|t|^{1+2\alpha} I_\varepsilon (t) \int_{SM} u(\pi\circ\phi_t(p,v)) \overline{\left( w(\pi\circ\phi_t(p,v)) - w(p)\right)} \dd\mathrm{L}(p,v)\dd t\\ 
&-&\int_{-\infty}^\infty b(t^2) |t|^{1+2\alpha} I_\varepsilon (t) \int_{SM} u(p)\overline{\left( w(\pi\circ\phi_t(p,v)) - w(p)\right)} \dd\mathrm{L}(p,v)\dd t.
\end{eqnarray*}
We now make a change of variable $(p,v) =\phi_{-t}(q,v')$ in the first integral and use the identity \eqref{flow invariance} we get
\begin{eqnarray*}\nonumber
N_\varepsilon &=& \int_{-\infty}^\infty b(t^2) |t|^{1+2\alpha} I_\varepsilon (t) \int_{SM} u(p) \overline{\left( w(p) - w(\pi\circ\phi_{-t}(p,v))\right)} \dd\mathrm{L}(p,y)\dd t\\ 
&-&\int_{-\infty}^\infty b(t^2)|t|^{1+2\alpha} I_\varepsilon (t) \int_{SM} u(p)\overline{\left( w(\pi\circ\phi_t(p,v)) - w(p)\right)} \dd\mathrm{L}(p,v)\dd t.
\end{eqnarray*}
The two integrals are essentially identical except that the time is reversed in the first integral. This can be taken care of by a substitution $t\mapsto -t$ to get
$$N_\varepsilon =2 \int_{-\infty}^\infty b(t^2) |t|^{1+2\alpha} I_\varepsilon (t) \int_{SM} u(p) \overline{\left( w(p) - w(\pi\circ\phi_{t}(p,v))\right)} \dd\mathrm{L}(p,v)\dd t.$$
Now use the splitting \eqref{split the measure} we have that $N_\varepsilon$ writes
$$2 \int\limits_{p\in M}u(p)\int_{-\infty}^\infty b(t^2) |t|^{1+2\alpha} I_\varepsilon (t)\int\limits_{v\in S_pM} \overline{\left( w(p) - w(\exp_p(tv))\right)} dS_p^{n-1}(v)\dd t{\rm dvol}_g(p).$$
In particular, by Lemma \ref{A in polar}, one obtains
$$N_\varepsilon = - 4 \int_M u(p) \overline{{\mathscr{A}}_{b,\varepsilon} w(p)}{\rm dvol}_g(p).$$ 
If $w\in C^2$, one can show that ${\mathscr{A}}_{b,\varepsilon} w(p) \to {\mathscr{A}}_b w(p) $ pointwise by Taylor expanding $w(\exp_p(v))$ near $v= 0$ and use the fact that 
$$\int_{\theta \in S^{n-1}} \theta\cdot \hat n \, \dd S^{n-1}(\theta) = 0$$ for all $\hat n\in S^{n-1}$ (in fact this is how the principal value integral of \eqref{def of A} is defined). As $M$ is compact we can use the same argument to get that for all $p\in M$ and $\varepsilon >0$, $|{\mathscr{A}}w(p)| \leqslant C \|w\|_{C^2(M)}$. So dominated convergence allows us to pass the limit
$$\lim_{\varepsilon \to 0} N_\varepsilon  = -4 \int_M u(p) \overline{{\mathscr{A}}_b w(p)}{\rm dvol}_g(p).$$
This combined with \eqref{L limit} shows that $\mathscr{A}_b$ is formally given by the Dirichlet form ${\mathscr{D}}_b$. The fact that the nullspace of ${\mathscr{D}}_b$ consists of only constants comes from \eqref{kernel of D is trivial}.
\end{proof}

\subsection{Generator on the torus}\label{subsec:generatortorus}
In this section, we prove Theorem \ref{torus} holds, that is, $-\mathscr{A}$ coincides with the fractional Laplacian $(-\Delta)^\alpha$ on $\T^n = \R^n / \Z^n$, which is defined on the Fourier basis by
$$
(-\Delta)^\alpha e_k = |2 \pi k|^{2\alpha} e_k, \quad k \in \Z^n,
$$
where $e_k : x \mapsto \exp 2 i \pi \langle k, x \rangle.$
\begin{proof}[Proof of Theorem \ref{torus}]
If $M = \T^n$, the generator of the L\'evy process is given by
\begin{equation}\label{eq:Af(x)}
\mathscr{A} f(x) = \dfrac{C(n,\alpha) }{2} \int_{\R^n} \left(\tilde f(x+v) + \tilde f(x-v) - 2\tilde f(x)\right) \dd v/|v|^{n+2\alpha}.
\end{equation}
Here, we set $\tilde f = f \circ \pi_{\T^n}$ where $\pi_{\T^n}: \R^n \to \R^n / \Z^n$ is the natural projection. Next, it follows from \cite[Lemma 5.1]{stinga2010extension} that for every Schwartz function $g \in \mathscr{S}(\R^n),$
\begin{equation}\label{eq:formuladelta(x)}
(-\Delta_{\R^n})^\alpha \tilde g(x) = -\dfrac{C(n,\alpha) }{2} \int_{\R^n} \left(g(x+v) + g(x-v) - 2g(x)\right) \dd v/|v|^{n+2\alpha},
\end{equation}
where $(-\Delta_{\R^n})^{\alpha} : \mathcal{S}(\R^n) \to \mathcal{S}(\R^n)$ is the fractional Laplacian on $\R^n$, defined by
\begin{equation}\label{eq:defdeltaalpha}
\mathcal{F}\bigl({(-\Delta_{\R^n})^{\alpha}g}\bigr)(\xi) = |\xi|^{2 \alpha} \mathcal{F}(g)(\xi), \quad g \in \mathcal{S}(\R^n).
\end{equation}
Here $\mathcal{F}$ is the Fourier transform. Since $(-\Delta_{\R^n})^\alpha$ is formally self-adjoint, it extends as an operator from $\mathcal{S}'(\R^n)$ to itself. Moreover, the right hand-side of \eqref{eq:formuladelta(x)} defines a continuous operator  $C^\infty_b(\R^n) \to C(\R^n)$ and by density of $\mathcal{S}(\R^n)$ in $\mathcal{S}'(\R^n)$, this implies
\begin{equation}\label{eq:af=-delta}
-\mathscr{A} f = (-\Delta_{\R^n})^\alpha\tilde f
\end{equation}
 by \eqref{eq:Af(x)}. Now we write $f = \sum_{k \in \Z^n} c_k e_k$ so that
\begin{equation}\label{eq:calFtildef}
\mathcal F{\tilde f} = \sum_k c_k \delta_{2\pi k}
\end{equation}
where the equality holds in $\mathscr{S}'(\R^n)$ and for $\xi \in \R^n$, $\delta_\xi$ is the Dirac distribution at $\xi$. Then
$$
-\mathscr{A} f = \sum_{k \in \Z^n} |2\pi k|^{2\alpha} c_k e_k
$$
by \eqref{eq:af=-delta} and \eqref{eq:calFtildef}, and the right-hand side of this equality is precisely $(-\Delta)^\alpha f$, which concludes the proof.
\end{proof}

%
%
%
%
%
%
%
%
%

%
%
%
%

\subsection{Generator on Anosov manifolds}
In this section we prove Theorem \ref{anosov generator} and deduce some of their consequences. For the remainder of this section we assume that $(M,g)$ is a Riemannian manifold whose geodesic flow is Anosov as defined in \S\ref{anosov overview}.

\subsubsection{Smoothing properties of averaging along the geodesic flow}
Let $a\in C^\infty(\R)$ be supported in $( 0, +\infty)$ and assume that all derivatives $a^{(k)}$, $k\geqslant 0$, are integrable. We define the operator $R_a : C^\infty(SM) \to \mathcal D'(SM)$ by
$$(R_a f)(z) := \int_0^\infty f(\phi_t(z)) a(t)\dd t.$$ 
We will prove

\begin{thm}\label{thm:averaging-smoothing}
Assume that $(M,g)$ is as above an Anosov manifold. Let $a\in C^\infty(\R)$ be supported in $( 0, +\infty)$. Also assume that all derivatives $a^{(k)}$, $k\geqslant 0$, are integrable. Then $\pi_\ast R_a \pi^\ast \in \Psi^{-\infty}(M)$ is a well defined smoothing operator on $\mathcal{D}'(M)$. 
\end{thm}

In fact, the exponential decay of correlations for Anosov manifolds \cite{Nonnenmacher-Zworski-decay}, implies that the result is also true in the case that $a(t) = e^{\varepsilon t} b(t)$, where $\varepsilon$ is small enough, and $b$ has all its derivatives bounded. Before we start the proof proper, we recall the following classical result of hyperbolic dynamics. Let $\varepsilon>0$ be fixed small enough. For $z\in SM$, let $U_z$ be the ball of radius $\varepsilon$ centered at $z$ of $W^{s}_{\mathrm{loc}}(z)\subset SM$. The following holds:

\begin{lem}\label{lemma:regularity-weak-stable-leaves}
Let $f\in C^\infty(SM)$. Then, the map
\[
f_z : U_z\times (-1,+\infty) \owns (w,\tau) \mapsto f(\phi_\tau(w))
\]
is smooth, with $C^k$ bounds ($k\geqslant 0$), independent of $z$. More precisely, the following holds: fix $z_0 \in SM$ and consider a family of vector fields $S_1,...,S_{n-1} \in C^\infty_{\mathrm{cs}}(SM,E_s)$ spanning locally $E_s$ near $z_0$. Then for all $\beta \in \N^{n-1}, j \geqslant 0$,
\[
\sup_{z \in W^{u}_{\mathrm{loc}}(z_0)}  \sup_{(w,\tau) \in U_z \times (-1,+\infty)} |S^\beta_w \partial^j_\tau f_z (w,\tau)| \leqslant C(\alpha,j) < \infty.
\]
\end{lem}
We refer the reader to \cite[\S 2]{DeLaLlave-Marco-Moriyon-86} for more details.

\begin{proof}
 We can now turn to the proof, divided into three steps. \\

\emph{Step 1: the operator $\pi_\ast R_a \pi^\ast : C^\infty(M) \to C^\infty(M)$ is continuous.} First, using H\"ormander's wavefront set rules \cite[\S 8.2]{hormander}, one has that for $f\in \mathcal{D}'(M)$, $\WF( \pi^\ast f) \subset \bV^\perp$ (that is, the pullback of a function is constant in the fibers of $SM$, hence smooth in the direction of the fibers -- in other words, its singularities are conormal to it), and for $u\in \mathcal{D}'(SM)$,
\begin{align*}
\WF( \pi_*u ) 		\subset \{ (p,\xi) \in T^*M \setminus \{0 \} \ |\ \exists v \in S_pM,\ (p,v, \underbrace{\dd_p\pi^{\top}(\xi)}_{\in \bV^\perp} ) \in \WF(u)\},
\end{align*}
(that is, the singularities in the direction of the fibers are killed by integrating over it).
Hence, it suffices to show that for $f\in C^\infty(M)$, we have $\WF(R_a\pi^*f) \subset E_s^*$. Indeed, this would imply that $\pi_* R_a \pi^* f \in C^\infty(M)$ according to \eqref{vertical does not intersect stable and unstable}. 

For this, we can invoke Lemma \ref{lemma:wf}: its content is that $\WF(R_a\pi^*f) \subset E_s^*$ provided we can prove that for each $z\in SM$, the restriction of $R_a\pi^*f$ to a piece of local weak stable manifold $W^{s0}_{\mathrm{loc}}(z)$ is smooth, with derivatives bounded uniformly in $z$. Then, we observe that for any point $z$, we can rewrite
\[
(R_a \pi^\ast f)_z(w,\tau) = \int_0^{+\infty} (\pi^\ast f)_z(w, \tau+t) a(t)\dd t. 
\]
Differentiating under the integral and using Lemma \ref{lemma:regularity-weak-stable-leaves}, we deduce that $R_a \pi^\ast f$ is indeed smooth along every weak-stable leaf, uniformly in the leaf. \\

\emph{Step 2: the operator $\pi_\ast R_a \pi^\ast : \mathcal{D}'(M) \to \mathcal{D}'(M)$ is continuous with respect to the weak* topology on $\D'(M)$.} Since $\pi_\ast R_a \pi^\ast : C^\infty(M) \to C^\infty(M)$ is bounded, the operator 
\[
(\pi_\ast R_a \pi^\ast)^\top : \mathcal{D}'(M) \to \mathcal{D}'(M)
\]
is also bounded. But $\pi_\ast R_a \pi^\ast$ is formally self-adjoint, that is, $(\pi_\ast R_a \pi^\ast)^\top = \pi_\ast R_a \pi^\ast$, which proves the claim. \\

\emph{Step 3: the operator $\pi_\ast R_a \pi^\ast : \mathcal{D}'(M) \to C^\infty(M)$ is bounded.} Equivalently, this means that the kernel of $\pi_\ast R_a \pi^\ast$ is smooth, or that its wavefront set is empty. In semi-classical terms, using the characterization \eqref{equation:decay}, this is equivalent to proving that for any smooth functions $S_j$, $\chi_j$, $j=1,2$, on $M$, with 
\[
dS_j \neq 0  \text{ on }\mathrm{supp}(\chi_j),
\]
we have as $h\to 0$
\[
\langle e^{iS_1/h} \chi_1, \pi_\ast R_a \pi^\ast ( e^{iS_2/h} \chi_2)\rangle = \mathcal{O}(h^\infty). 
\]
Rewriting this as 
\[
\langle e^{i\tilde S_1/h} \tilde \chi_1,  R_a ( e^{i\tilde S_2/h} \tilde \chi_2)\rangle = \mathcal{O}(h^\infty),
\]
with $\tilde S_j = \pi^\ast S_j$, and $\tilde \chi_j =  \pi^\ast \chi_j$, $j=1,2$, and using a partition of unity, it suffices to prove the following. 
\begin{lem}\label{lemma:WF-R_f}
For $j =1,2$, given any $z_{j}\in SM$, for any smooth $\psi_{j}, \theta_{j}$ supported in a small ball near $z_{j}$, with $d\theta_j \in \bV^\perp \setminus \{0\}$ on the support of $\psi_j$, we have the estimate
\[
\langle e^{i\theta_1/h} \psi_1,  R_a   ( e^{i\theta_2/h} \psi_2)\rangle = \mathcal{O}(h^\infty)
\]
\end{lem}

\begin{proof}[Proof of Lemma \ref{lemma:WF-R_f}]
We start by observing that since $d\theta_j\in \bV^\perp\setminus \{0\}$, for some $\varepsilon>0$, 
\begin{equation}\label{eq:positive}
\| (d\theta_1)_{|{E_s}\oplus E^0} \| \geqslant \varepsilon. 
\end{equation}
In particular, for some $T>0$, for all $t>T$
\begin{equation}\label{eq:sufficiently-close-to-Eu*}
\| d(\theta_2\circ \phi_t)_{|{E_s}} \| \leqslant \frac{1}{2} \| (d\theta_1)_{|{E_s}\oplus E^0} \|.
\end{equation}
Let $\eta\in C^\infty_c(\R^+)$ be a cutoff function, with $\eta= 1$ in $[0,T]$, and $\eta=0$ in $[T+1,+\infty)$. We can cut $R_a = R_{a\eta} + R_{a(1-\eta)}$. 

For the $R_{a\eta}$ term, we write the contribution to our scalar product
\[
\begin{split}
\langle e^{i\theta_1/h} \psi_1, & R_{a\eta}  ( e^{i\theta_2/h} \psi_2)\rangle \\
		&= \int_0^{T+1} a(t)\eta(t) \int_{SM} e^{(i/h)(\theta_2\circ\phi_t - \theta_1)} \psi_1\psi_2\circ\phi_t 
\end{split}
\]
This is an oscillatory integral, with phase $\theta_2\circ\phi_t - \theta_1$. It is stationary in the $SM$ variable if $d\theta_2 \circ d\phi_t = d\theta_1$. However, since $\phi_t(\bV^\perp) \cap \bV^\perp \cap (E^0)^\perp = \{0\}$ by \eqref{equation:cp}, this is only possible at a point where both $d\theta_1$ and $d\theta_2$ belong to $E^0_\ast$. In such a case, the phase is non stationary in the $t$ variable. It follows from usual non-stationary phase estimates that 
\[
\langle e^{i\theta_1/h} \psi_1, R_{a\eta}  ( e^{i\theta_2/h} \psi_2)\rangle = \mathcal{O}(h^\infty).
\]
Here it was important that we used all variables (stable, unstable and time). 

Let us now concentrate on the $R_{a(1-\eta)}$ contribution. According to \eqref{eq:sufficiently-close-to-Eu*}, it is now sufficient to integrate in the weak stable direction. We can use the absolute continuity of the weak stable foliation discussed at the end of \S\ref{anosov overview}. Using \eqref{equation:desint}, we can write
\begin{equation}
\label{big integral}
\begin{aligned}
&\langle e^{i\theta_1/h} \psi_1,  R_{a(1-\eta)}  ( e^{i\theta_2/h} \psi_2)\rangle \\
&\hspace{2cm} = \int_{z\in W^{u}_{\mathrm{loc}}(z_1)}  \Biggl(\underset{=I_z}{\underbrace{\int_{W^{s0}_{\mathrm{loc}}(z)} \overline{e^{i\theta_1/h} \psi_1}  R_{a(1-\eta)}   ( e^{i\theta_2/h} \psi_2) \dd\mathcal{L}_z}}\Biggr) \dd \mu(z),\\
\end{aligned}
\end{equation}
where $\mathcal{L}_z$ is a smooth volume measure along $W^{s0}_{\mathrm{loc}}(z)$ and $\mu$ is a smooth measure along $W^u_{\mathrm{loc}}(z_1)$. We now work directly on $W^{s0}_{\mathrm{loc}}(z)$ and introduce the coordinates $\kappa_z(w,\tau)  = \phi_\tau(w) \in W^{s0}_{\mathrm{loc}}(z)$ for $(w,\tau) \in U_z \times (-\eps, \eps)$. Note that we can further identify $U_z$ with an open subset $U \subset \R^{n-1}$. We have:
\[
\begin{split}
I_z = \int\limits_{U\times(-\varepsilon,\varepsilon)\times \R^+}\hspace{-25pt}  e^{\frac{i}{h}\left( \theta_2\circ\kappa_z(w,\tau+t)- \theta_1\circ\kappa_z(w,\tau)\right)}& a(t)(1-\eta(t)) \\
			&\hspace{-30pt} \psi_2\circ\kappa_z(w,\tau+t)\overline{\psi_1\circ\kappa_z(w,\tau)} \rho_z(w,\tau) \dd w \dd\tau\dd t. 
\end{split}
\]
We observe that the above expression for $I_z$ is of the form
$$I_z = \int\limits_{U\times(-\varepsilon,\varepsilon)\times \R^+} e^{i\Psi(w,\tau,t) /h} b(w,\tau,t)\dd w \dd\tau\dd t$$
which is still an oscillatory integral, with phase 
$$\Psi =\left( \theta_2\circ\kappa_z(w,\tau+t)- \theta_1\circ\kappa_z(w,\tau)\right)$$
 and smooth amplitude $b(w,\tau, t)$ which is supported on $\{(w,\tau, t)\mid t \geqslant T\}$. Note that, although not indicated in the notation, $\Psi$ and $b$ both depend on $z$ in a continuous fashion (and similarly for all their derivatives, as follows from \eqref{equation:desint}). Moreover, by Lemma \ref{lemma:regularity-weak-stable-leaves}, all the derivatives of $\Psi$ are uniformly bounded on $U\times(-\varepsilon,\varepsilon)\times \R^+$ and all the derivatives of $b$ are smooth and integrable on $U\times(-\varepsilon,\varepsilon)\times \R^+$. 
Combining \eqref{eq:positive} and \eqref{eq:sufficiently-close-to-Eu*}, we see that for some $\delta>0$, 
\[
\| (d\theta_2\circ d\phi_t)_{|E^0}\| > \delta, \text{ or } \| (d\theta_2 \circ d\phi_t -d\theta_1)_{|{E_s}\oplus E^0} \|  > \delta, 
\]
implying $| d\Psi | > \delta$. We then apply a usual non-stationary phase argument to conclude that $I_z = O(h^\infty)$ and therefore by \eqref{big integral}
$$\langle e^{i\theta_1/h} \psi_1,  R_{a(1-\eta)}  ( e^{i\theta_2/h} \psi_2)\rangle = O(h^\infty).$$
\end{proof}
This concludes the proof of Theorem \ref{thm:averaging-smoothing}.
\end{proof}

\subsubsection{Proof of Theorem \ref{anosov generator} }

\begin{proof}[Proof of Theorem \ref{anosov generator}]
For property i), let $r_{\mathrm{inj}}$ be the injectivity radius of $(M,g)$ and let $\chi(t) = 1$ for $|t|< r_{\mathrm{inj}}^2/4$ and $\chi(t)=0$ for $|t|> r_{\mathrm{inj}}^2/2$. We then have that $\mathscr{A} = \mathscr{A}_1 + \mathscr{A}_2$ where
{\begin{eqnarray}
\label{mA1 def}
\mathscr{A}_1u(p):=C(n,\alpha) \mathrm{p.v.}\int\limits_{v\in T_pM\setminus 0} \chi(|v|_g^2) \frac{\left( u(\exp_p(v)) - u(p)\right)}{|v|_g^{n+2\alpha}} \dd T_p(v)
\end{eqnarray}}
{\begin{eqnarray}
\label{mA2 def}
\ \mathscr{A}_2u(p):=C(n,\alpha)\int\limits_{v\in T_pM\setminus 0}\left(1- \chi(|v|_g^2) \right)\frac{\left( u(\exp_p(v)) - u(p)\right)}{|v|_g^{n+2\alpha}} \dd T_p(v)
\end{eqnarray}}
where $\dd T_p(v)$ is the volume form on $T_pM$ for the metric $g_p$. Note that both $\mathscr{A}_1$ and $\mathscr{A}_2$ are formally self-adjoint due to Proposition \ref{uniqueness'}.

For $\mathscr{A}_1$, since $|v|^2_g< r^2_{inj} /2$ we can perform a change of coordinate $q = \exp_p(v)$ so that 
{ \begin{eqnarray}
\label{A1 in coord}
\mathscr{A}_1u(p)= \mathrm{p.v.}\int_M \chi(\mathrm{dist}_g(p,q)^2) \frac{\left( u(q) - u(p)\right)}{\mathrm{dist}_g(p,q)^{n+2\alpha}}  J(p,q){\rm dvol}_g(q)
\end{eqnarray}}where $J(p,q)$ is the Jacobian determinant of the map $q\mapsto \exp_p^{-1}(q)$. Note that $J(p,q)$ is jointly smooth for $\mathrm{dist}_g(p,q)^2 <r^2_{inj}/2$ and $J(p,p) = 1$.

For $\mathscr{A}_2$, use polar coordinates to write
\begin{eqnarray}
\label{A2 in polar}\nonumber
\mathscr{A}_2u(p)&=& \int_{v\in S_pM}\int_{0}^\infty\left(1- \chi(t^2) \right)\frac{\left( u(\exp_p(tv)) - u(p)\right)}{t^{1+2\alpha}}\dd t \dd S^{n-1}(v)\\
&=& \left(\pi_* R_{a} \pi^* u\right)(p) -  u(p) \left( \pi_* R_{a} \pi^*1\right)(p)
\end{eqnarray}
for $a(t) = (1-\chi(t^2)) t^{-1-2\alpha}$.
By Theorem \ref{thm:averaging-smoothing}, $\pi_* R_{a} \pi^*: \D'(M) \to C^\infty(M)$ so the proof of Theorem \ref{anosov generator} is complete.
\end{proof}

We conclude this subsection with a statement about the microlocal properties of $\mathscr{A}^+$. 
\begin{lem}
\label{singular part of A+}
Let $\mathscr{L}\in \Psi^{-2\alpha}_{\mathrm{cl}}(M)$ be the operator given by the integral kernel 
\begin{equation}\label{eq:integralkernelparametrix}
-C(n,-\alpha)\frac{\chi(\mathrm{dist}_g(p,q)^2)}{   \mathrm{dist}_g(p,q)^{n-2\alpha}}
\end{equation}
where $\chi\in C^\infty_c(\R)$ is $1$ near $0$ and $0$ outside of $(-r_{\mathrm{inj}}^2/2, r_{\mathrm{inj}}^2/2)$. Then 
$$
\mathscr{A}^+ = \mathscr{L} + \mathscr{R}
$$
with $\mathscr{R} \in \Psi_{\mathrm{cl}}^{-2\alpha -1}(M) + \Psi^{-4\alpha}_{\mathrm{cl}}(M)$ .
\end{lem}

\begin{proof}
The operator $\mathscr{A}$ has Schwartz kernel given by \eqref{anosov kernel} so is therefore in $\Psi^{2\alpha}_{\mathrm{cl}}(M) + \Psi^0_{\mathrm{cl}}(M)$ with principal symbol $\sigma_{2\alpha}(\mathscr{A}) =-c(n,\alpha)|\eta|_g^{2\alpha}$. So by standard parametrix construction \cite{taylor2013partial} one has that $\mathscr{A}^+\in \Psi^{-2\alpha}_{\mathrm{cl}}(M) + \Psi^{-4\alpha}_{\mathrm{cl}}(M)$ with principal symbol $- |\eta|_g^{-2\alpha}$. The principal symbol of $\mathscr{L} \in \Psi_{\mathrm{cl}}^{-2\alpha}(M)$ is also $-|\eta|_g^{-2\alpha}$ as it can be verified by a direct computation. So we have that $\mathscr{R} :=  \mathscr{A}^+-\mathscr{L} \in \Psi_{\mathrm{cl}}^{-2\alpha -1}(M) + \Psi^{-4\alpha}_{\mathrm{cl}}(M)$.
\end{proof}

\subsection{Generator on the sphere}

Compared to that of Anosov manifolds, the case of the sphere is easier to compute but yields an operator which is more complicated.

\subsubsection{Proof of Theorem \ref{sphere}}
This is given by direct computation. Let $\chi\in C^\infty_c(\R)$ take the value $1$ when $|t|\leqslant \left(\pi/8\right)^2$ and value $0$ for $|t| \geqslant \left(\pi/4\right)^2$ and define $\chi_1, \chi_2,\chi_3\in C^\infty(\mS^n\times \mS^n)$ by
\begin{eqnarray}
\label{def of chi1 chi2 chi3}
\quad \chi_1(p,q) := \chi(\mathrm{dist}_g(p,q)^2),\quad \chi_2(p,q) := \chi(\mathrm{dist}_g(p,-q)^2), \quad p,q \in \mS^n,
\end{eqnarray}
and $\chi_3 := 1- \chi_1 -\chi_2.$
We now write 
\begin{equation}
\begin{aligned}
\label{Au split in 3 parts}
C(n,\alpha)^{-1}\mathscr{A} u(p) &=  \mathrm{p.v.}\int_{v\in T_pM\setminus 0} \chi_1(p,\exp_p(v)) \frac{ \left( u(\exp_p(v) - u(p)\right)}{|v|_g^{n+2\alpha}} \dd T_p(v)\\
&+\int_{v\in T_pM\setminus 0} \chi_2(p,\exp_p(v)) \frac{ \left( u(\exp_p(v) - u(p)\right)}{|v|_g^{n+2\alpha}} \dd T_p(v) \\
&+\int_{v\in T_pM\setminus 0} \chi_3(p,\exp_p(v)) \frac{ \left( u(\exp_p(v) - u(p)\right)}{|v|_g^{n+2\alpha}} \dd T_p(v).
\end{aligned}
\end{equation}
For $j=1,2,3$, denote by $\mc{I}_j(p)$ the integral in the right-hand side of \eqref{Au split in 3 parts} involving $\chi_j$. Observe that $\chi_1(p,\exp_p(v)) = 0$ whenever
$$|v|_g\notin [0, \pi/4]\cup \bigcup_{k = 1}^\infty [2\pi k - \pi/4,2\pi k + \pi/4].$$ So
for $\mc{I}_1u(p)$, we write, using spherical coordinates on $T_pS^n$,
\begin{eqnarray*}
\mc{I}_1 u(p) &=& \mathrm{p.v.} \int\limits_{v\in S_p(\mS^n)}\int_0^\infty \chi_1\left(p,\exp_p(tv)\right) \frac{\left(u(\exp_p(tv)) - u(p)\right)}{t^{1+2\alpha}}\dd t \dd{S^{n-1}}(v)\\
&=&\mathrm{p.v.}\int\limits_{v\in S_p(\mS^n)}\int_0^{\pi/4} \chi_1\left(p,\exp_p(tv)\right) \frac{\left(u(\exp_p(tv)) - u(p)\right)}{t^{1+2\alpha}}\dd t \dd{S^{n-1}}(v)\\
&+& \sum\limits_{k = 0}^\infty \int\limits_{v\in S_p(\mS^n)}\int\limits_{7\pi/4+2\pi k}^{9\pi/4+2\pi k} \chi_1\left(p,\exp_p(tv)\right) \frac{\left(u(\exp_p(tv)) - u(p)\right)}{t^{1+2\alpha}}\dd t \dd{S^{n-1}}(v).
\end{eqnarray*}
Observe that each individual term of $\mc{I}_1$ is formally self-adjoint by Proposition \ref{uniqueness'}. Shifting the $t$ integral via a change of variable we get
\begin{eqnarray*}
\mc{I}_1u(p) &=&\mathrm{p.v.}\int\limits_{v\in S_p(\mS^n)}\int_0^{\pi/4} \chi_1\left(p,\exp_p(tv)\right) \frac{\left(u(\exp_p(tv)) - u(p)\right)}{t^{1+2\alpha}}\dd t \dd{S^{n-1}}(v)\\
&+& \sum\limits_{k = 1}^\infty \int\limits_{v\in S_p(\mS^n)}\int\limits_{-\pi/4}^{\pi/4} \chi_1\left(p,\exp_p(tv)\right) \frac{\left(u(\exp_p(tv)) - u(p)\right)}{(2\pi k+t)^{1+2\alpha}}\dd t \dd{S^{n-1}}(v).
\end{eqnarray*}
Now we make the change of variable $ q = \exp_p(tv)$ so that $t = \sign(t) \mathrm{dist}_g(p,q)$ for $t\in (-1/4\pi, 1/4\pi)$, and $J(p,q) ={\rm det}\left( D\exp_p^{-1}\mid_q\right)$ is the Jacobian of the change of variable, we get
\begin{eqnarray}
\label{I integral final}\nonumber
\mc{I}_1u(p) &=&\mathrm{p.v.}\int_M \chi_1\left(p,q\right) \frac{\left(u(q) - u(p)\right)}{\mathrm{dist}_g(p,q)^{n+2\alpha}}J(p,q) {\rm dvol}_g(q)\\
&+& \sum\limits_{k = 1}^\infty \int_M \chi_1\left(p,q\right) \frac{\left(u(q) - u(p)\right)J(p,q)}{(2\pi k-\mathrm{dist}_g(p,q))^{1+2\alpha}(\mathrm{dist}_g(p,q))^{n-1}}{\rm dvol}_g(q)\\\nonumber
&+& \sum\limits_{k = 1}^\infty \int_M \chi_1\left(p,q\right) \frac{\left(u(q) - u(p)\right)J(p,q)}{(2\pi k+\mathrm{dist}_g(p,q))^{1+2\alpha}(\mathrm{dist}_g(p,q))^{n-1}}{\rm dvol}_g(q).
\end{eqnarray}
The sum converges absolutely. So the expression $C(n,\alpha) \mc{I}_1 u$ has leading term 
$$u\mapsto C(n,\alpha)\, \mathrm{p.v.}\int_M \chi_1\left(p,q\right) \frac{\left(u(q) - u(p)\right)}{\mathrm{dist}_g(p,q)^{n+2\alpha}}J(p,q) {\rm dvol}_g(q)$$
which is an elliptic classical pseudodifferential operator with principal symbol $-|\eta|^{2\alpha}$. So we see that the first term of \eqref{Au split in 3 parts} is 
$$C(n,\alpha)\mc{I}_1 = \mathscr{A}_{2\alpha} + \mathscr{A}_0$$ for some formally selfadjoint $\mathscr{A}_{2\alpha}\in \Psi^{2\alpha}_{cl}(M)$ and $\mathscr{A}_0\in \Psi^0_{cl}(M)$. With the principal symbol of $\mathscr{A}_{2\alpha}$ being $ -|\eta|_{g}^{2\alpha}$.

Denote by $-p$ the antipodal point of $p\in \mS^n$. Similar calculation as for the case of $\mc{I}_{1}$ using $\exp_{-p}(tv)$ for $v\in S_{-p}(\mS^n)$ yields that $\mc{I}_2$ and $\mc{I}_3$ are formally selfadjoint and given by
{
\begin{eqnarray}
\label{II integral final}\nonumber
\mc{I}_2u(p) &=&\sum\limits_{k=0}^\infty \int_M \chi_2\left(-p,q\right)\frac{\left(u(q) - u(p)\right)J(-p,q)}{(2\pi k+ \pi-\mathrm{dist}_g(-p,q))^{1+2\alpha}(\mathrm{dist}_g(-p,q))^{n-1}}{\rm dvol}_g(q)\\
&+& \sum\limits_{k=0}^\infty \int_M \chi_2\left(-p,q\right)\frac{\left(u(q) - u(p)\right)J(-p,q)}{(2\pi k+\pi+\mathrm{dist}_g(-p,q))^{1+2\alpha}(\mathrm{dist}_g(-p,q))^{n-1}}{\rm dvol}_g(q)
\end{eqnarray}}
and
\begin{eqnarray}
\label{III integral final}
\mc{I}_3u(p) &=& \int\limits_{v\in T_pM\setminus 0} \chi_3(p,\exp_p(v)) \frac{ \left( u(\exp_p(v) - u(p)\right)}{|v|_g^{n+2\alpha}} \dd T_p(v)\\\nonumber
&=& \int_M k(p,q) (u(q) - u(p)){\rm dvol}_g(q)
\end{eqnarray}
for some $k(\cdot,\cdot)\in C^\infty(M\times M)$ satisfying $k(-p,q) = k(p,-q)$.

The infinite sum in \eqref{II integral final} is absolutely convergent. Let $\mathscr{J}$ be the operator defined by pulling back by the antipodal map. A local coordinate calculation using normal coordinates yields that it is of the form $\mathscr{A}_0 u + \mathscr{J} \mathscr{A}_{-1}u$ for some $\mathscr{A}_{-1}\in \Psi^{-1}_{\mathrm{cl}}(M)$ formally selfadjopint with principal symbol $-c(n) |\eta|^{-1}$ and $\mathscr{A}_0\in \Psi^{0}_{\mathrm{cl}}(M)$.

Finally since $k(p,q)$ is smooth, \eqref{III integral final} is of the form $\mathscr{A}_{-\infty} u + \mathscr{A}_0 u$ for $\mathscr{A}_{-\infty}\in \Psi^{-\infty}(M)$ and $\mathscr{A}_0\in \Psi^{0}_{\mathrm{cl}}(M)$.

All operators commute with $\mathscr{J}$ since $\mathrm{dist}_g(-p,q) = \mathrm{dist}_g(p,-q)$, $J(p,q) = J(q,p)$, $J(-p,q) = J(p,-q)$ by symmetry of the sphere.

Inserting the conclusions about \eqref{I integral final}, \eqref{II integral final}, and \eqref{III integral final} into \eqref{Au split in 3 parts} we have the assertion of Theorem \ref{sphere}.

\subsubsection{Microlocal Properties of $\mathscr{A}$ and its inverse on $\mS^n$}
We now construct a (partial) parametrix for the operator $\mathscr{A}$:

\begin{lem}
\label{parametrix of A sphere}
There exists operators $\mathscr{K}\in \Psi^{-2\alpha}_{\mathrm{cl}}(\mS^n) + \Psi_{\mathrm{cl}}^{-4\alpha}(\mS^n)$ and 
$$\mathscr{B} \in \Psi_{\mathrm{cl}}^{-1-4\alpha}(\mS^n)+ \Psi_{\mathrm{cl}}^{-1-6\alpha}(\mS^n) + \Psi^{-1-8\alpha}_{\mathrm{cl}}(\mS^n)$$ such that 
\begin{equation}
\label{parametrix eq on sphere}
(\mathscr{K} + \mathscr{B} \mathscr{J})\mathscr{A} = I +\Psi_{\mathrm{cl}}^{-2-4\alpha}(\mS^n)+ \Psi_{\mathrm{cl}}^{-2-6\alpha}(\mS^n) + \Psi^{-2-8\alpha}_{\mathrm{cl}}(\mS^n).
\end{equation}
The operator $\mathscr{K}$ has principal symbol $-|\eta|_g^{-2\alpha}$ and $\mathscr{B} = -\mathscr{K} \mathscr{A}_{-1} \mathscr{K}$ has principal symbol $\tilde c(n,\alpha)|\eta|_g^{-1-4\alpha}$ for some nonvanishing constant $\tilde c(n,\alpha)$.
\end{lem} 
\begin{proof}
Let $\mathscr{K}\in \Psi^{-2\alpha}_{\mathrm{cl}}(\mS^n) + \Psi_{\mathrm{cl}}^{-4\alpha}(\mS^n)$ be a parametrix for the elliptic pseudodifferential operator $\mathscr{A}_{2\alpha} + \mathscr{A}_0$. We then have that $\mathscr{K}\in \Psi^{-2\alpha}_{\mathrm{cl}}(\mS^n) + \Psi_{\mathrm{cl}}^{-4\alpha}(\mS^n)$ with principal symbol $-|\eta|_g^{-2\alpha}$. Now set 
\[
\mathscr{B} = -\mathscr{K} \mathscr{A}_{-1} \mathscr{K}\in \Psi^{-1-4\alpha}(\mS^n)+ \Psi_{\mathrm{cl}}^{-1-6\alpha}(\mS^n) + \Psi^{-1-8\alpha}_{\mathrm{cl}}(\mS^n),
\]
with principal symbol $c(n,\alpha)|\eta|_g^{-1-4\alpha}$. We have, using the commuting property of $\mathscr{J}$ with $\mathscr{A}_{2\alpha}$, $\mathscr{A}_{0}$ and $\mathscr{A}_{-1}$,
\begin{align*}
(\mathscr{K} + \mathscr{B} \mathscr{J}) \mathscr{A} &= (\mathscr{K} + \mathscr{B} \mathscr{J}) (\mathscr{A}_{2\alpha} + \mathscr{A}_0 + \mathscr{J}\mathscr{A}_{-1})\\
\nonumber
										&= I + \mathscr{B} \mathscr{A}_{-1} + \Psi^{-\infty}(\mS^n).
\end{align*}
Simple book keeping asserts that 
\[
\mathscr{B} \mathscr{A}_{-1} \in \Psi^{-2-4\alpha}(\mS^n)+ \Psi_{\mathrm{cl}}^{-2-6\alpha}(\mS^n) + \Psi^{-2-8\alpha}_{\mathrm{cl}}(\mS^n)
\]
and we obtain \eqref{parametrix eq on sphere}. To compute the principal symbol of $\mathscr{B}$ one simply use standard principal symbol calculus and the fact that $\mathscr{A}_{-1}$ has principal symbol $c(n)|\eta|_g^{-1}$ as given in Theorem \ref{sphere}.
\end{proof}
We have the following microlocal structure for $\mathscr{A}^+$:
\begin{lem}
\label{microlocal structure of A+ sphere}
Let $\mathscr{K} + \mathscr{B}\mathscr{J}$ be the partial parametrix for $\mathscr{A}$ constructed in Lemma \ref{parametrix of A sphere}. Then 
$$\mathscr{A}^+ =\left(\mathscr{K} + \mathscr{B}\mathscr{J}\right)  + \mathscr{S}_1 + \mathscr{J} \mathscr{S}_2$$
where $\mathscr{S}_1$ is a finite sum of classical $\Psi$DOs of order $-2-6\alpha $ or less, $S_2$ is a sum of classical $\Psi$DOs of order $-3-8\alpha$ or less. 
\end{lem}
\begin{proof}
We write
{\begin{eqnarray}
\label{A+ minus parametrix}
\left(\mathscr{A}^+ -\left(\mathscr{K} + \mathscr{B}\mathscr{J}\right)\right)\mathscr{A} \in \Psi_{\mathrm{cl}}^{-2-4\alpha}(\mS^n)+ \Psi_{\mathrm{cl}}^{-2-6\alpha}(\mS^n) + \Psi^{-2-8\alpha}_{\mathrm{cl}}(\mS^n)
\end{eqnarray}}
Taking the adjoint of \eqref{parametrix eq on sphere} we have
$$\mathscr{A}(\mathscr{K}^* + \mathscr{J}^* \mathscr{B}^*) = I +\mathscr{M}$$ with
$$\mathscr{M} \in \Psi_{\mathrm{cl}}^{-2-4\alpha}(\mS^n)+ \Psi_{\mathrm{cl}}^{-2-6\alpha}(\mS^n) + \Psi^{-2-8\alpha}_{\mathrm{cl}}(\mS^n). $$
So if we hit \eqref{A+ minus parametrix} on the left with $(\mathscr{K}^* + \mathscr{J}^* \mathscr{B}^*)$, we get that 
$$\left(\mathscr{A}^+ - \left(\mathscr{K} + \mathscr{B}\mathscr{J}\right)\right)\left( I +\mathscr{M}\right) = \mathscr{S}_1 + \mathscr{J} \mathscr{S}_2$$
where $\mathscr{S}_1$ is a finite sum of classical $\Psi$DOs of order $-2-6\alpha $ or less, $\mathscr{S}_2$ is a sum of classical $\Psi$DOs of order $-3-8\alpha$ or less. Hit the above expresson on the right by the parametrix of $(I + \mathscr{M})$ we have the result.
\end{proof}
\begin{cor}
\label{A+ = L + R sphere}
Let $\mathscr{L}\in \Psi^{-2\alpha}_{\mathrm{cl}}(M)$ be the operator given by the integral kernel \eqref{eq:integralkernelparametrix} where $\chi\in C^\infty_c(\R)$ is $1$ near $0$ and $0$ outside of $(-r_{\mathrm{inj}}^2/2, r_{\mathrm{inj}}^2/2)$. Then 
$$
\mathscr{A}^+ - \mathscr{L} = \mathscr{S}_1 + \mathscr{J} \mathscr{S}_2
$$
where $\mathscr{S}_2$ is some pseudodifferential operator and $\mathscr{S}_1$ is an element in $\Psi^{-1-2\alpha}_{\mathrm{cl}}(M)$ plus a finite sum of classical pseudodifferential operators each of which is of order $-4\alpha$ or less.
\end{cor}

\subsection{Mapping properties of the generator}
In this paragraph we prove Corollary \ref{sphere cor} about the mapping properties of $\mathscr{A}$. 
\begin{proof}[Proof of Corollary \ref{sphere cor}]
For \emph{\ref{item:1cor}}, we first show that the null-space of $\mathscr{A}$ in $W^{s,m}(\mS^n)$ consists of only constants. Suppose $\mathscr{A} u = 0$ for some $u\in W^{s,m}(\mS^n)$ with $s\in \R$. Then apply Lemma \ref{parametrix of A sphere} we get that $(I + K) u = 0$ for some $K : W^{s,m}(\mS^n) \to W^{s+2+4\alpha,m}(\mS^n)$ for all $s\in \R$. This means that $u\in C^\infty(\mS^n)$. By Proposition \ref{uniqueness}, $u$ is constant.

To see that $\mathscr{A}: W^{s,m}(\mS^n)\to W^{s-2\alpha,m}(\mS^n)$ has finite dimensional cokernel, simply observe that $\mathscr{A}$ is formally self-adjoint and use the fact that the null-space of $\mathscr{A}$ consists of only constants.

To prove the assertion about discreteness of spectrum, simply observe that both $K$ and the resolvent of $\mathscr{A}$ are both compact operators from $W^{s,m}(M)\to W^{s,m}(M)$. The fact that eigenfunctions are smooth comes from the partial parametrix constructed in Lemma \ref{parametrix eq on sphere}.

For \emph{\ref{item:2cor}}, the construction of the operator $\mathscr{A}^+$ follows from \emph{\ref{item:1cor}} and basic functional analysis.
\end{proof}

\section{Expected Stopping Time for Random Search}\label{sec:stoppingtime}

The aim of this section is to prove Theorem \ref{narrow capture} and Theorem \ref{narrow capture sphere}.

\subsection{Fundamental properties of the expected stopping time}

Recall that, if $\Omega \subset \R^n$ is a bounded smooth domain, $(X_t)_{t \geqslant 0}$ denotes the Brownian motion and $\tau$ denotes the exit time of $\Omega$ for this process, then the expected exit time
\[
u(p) := \mathbb E(\tau ~|~ X_0 = p)
\]
solves the equation $\Delta u = -1$ in $\Omega$ with Dirichlet boundary condition $u|_{\partial \Omega} = 0$, see \cite[Chapter 11]{taylor2013partial} for instance. We claim that the expected stopping time  for our L\'evy flight satisfies similar properties as the ones connecting the exit time of the Brownian motion to the Dirichlet Laplacian.  

\begin{prop}
\label{proposition uepsilon}
If $(M,g)$ is the round sphere, the flat torus, or Anosv, the expected stopping time satisfies the following properties ($1<m<1/\alpha$):
\begin{enumerate}[label=\emph{(\roman*)}]
\item In the Anosov and flat torus case,
\[
u_\varepsilon \in L^\infty(M) \cap C^\infty(M \setminus \partial B_\varepsilon(p_0)) \cap W^{2\alpha,m}(\overline{\Omega_\varepsilon}),
\]
while in the sphere case 
\[
u_\varepsilon \in L^\infty(M) \cap C^\infty(M \setminus \partial B_\varepsilon(\pm p_0)) \cap W^{2\alpha,m}(\overline{\Omega_\varepsilon}),
\]
for all $\varepsilon > 0$.
\item One has the fundamental relation:
\begin{equation}
\label{equation uepsilon}
\mathscr{A} u_\varepsilon = - 1, \text{ on } \Omega_\varepsilon := M \setminus \overline{B_\varepsilon(p_0)},  \qquad u_\varepsilon = 0,  \text{ on } \overline{B_\varepsilon(p_0)}.
\end{equation}
\end{enumerate}
\end{prop}

This is very similar to the equation satisfied by the expected exit time for the Brownian motion. Note, however, that due to the nonlocality of the generator $\mathscr{A}$, the boundary condition $u = 0$ on $\partial \Omega$ in the Laplacian case has to be replaced here by $u_\varepsilon = 0$ on $\overline{B_\varepsilon(p_0)}$. This will actually create a lot of troubles in the proofs and showing \eqref{equation uepsilon} will actually require some effort.

Using \eqref{equation uepsilon}, we can introduce $F_\varepsilon \in \mc{D}'(M)$ such that
\begin{equation}
\label{equation defF}
\mathscr{A} u_\varepsilon = F_\varepsilon - \mathbf{1}_{\Omega_{\varepsilon}}.
\end{equation}
By construction, $F_\varepsilon$ is a distribution supported in $\overline{B_\varepsilon(p_0)}$. An important idea will be to study the properties of $F_\varepsilon$ (and not that of $u_\varepsilon$ directly), and then to deduce properties from \eqref{equation defF} properties for the expected stopping time $u_\varepsilon$.

Before stating the result, we need to introduce some notation. First, we introduce rescaled geodesic coordinates centred at $p_0$: let $E_1,\dots, E_n\in T_{p_0}M$ be a orthonormal basis and define $\psi_\varepsilon :{\B^n} \to B_\varepsilon(p_0)$ by
\begin{eqnarray}
\label{rescaled geod coord}
\psi_\varepsilon (x) := \exp_{p_0}(\varepsilon x_1 E_1 + \dots+\varepsilon x_nE_n).
\end{eqnarray}
For $m \in [1,\infty]$, recall from \S\ref{sssection:functional-spaces} that $\dot L^m (\overline {B_\varepsilon(p_0)})$ denotes the space of functions $f$ such that $f \in L^m(M)$ and $\mathrm{supp}(f) \subset \overline {B_\varepsilon(p_0)}$.

The following holds:

\begin{prop}
\label{proposition fepsilon}
The distribution $F_\varepsilon \in \mc{D}'(M)$ satisfies:
\begin{enumerate}[label=\emph{(\roman*)}]
\item $\mathrm{supp}(F_\varepsilon) \subset  \overline{B_\varepsilon(p_0)}$ and $F_\varepsilon \in C^\infty(M \setminus \partial B_\varepsilon(p_0))$,
\item $u_\varepsilon = \mathscr{A}^+(F_\varepsilon - \mathbf{1}_{\Omega_{\varepsilon}}) + C_\varepsilon$ where
\begin{equation}
\label{expansion cepsilon}
\begin{split}
C_\varepsilon & := |M|^{-1} \int_M u_\varepsilon(p) {\rm dvol}_g(p)  = \varepsilon^{2\alpha - n} |M| c(n,\alpha)(1+\mc{O}(E(\alpha,\varepsilon))),
\end{split}
\end{equation}
where $c(n,\alpha)$ is given by \eqref{equation:cna} and the error term $E(\alpha,\varepsilon)$ is
\begin{eqnarray}
\label{Ealphaeps}
E(\alpha, \varepsilon) = 
\begin{cases}
\varepsilon^{2\alpha}, \quad &\text{if} \ \alpha < 1/2, \\
  \varepsilon|\log\varepsilon|, \quad &\text{if} \ \alpha = 1/2,\\
\max( \varepsilon, \varepsilon^{n-2\alpha}), \quad &\text{if} \ \alpha > 1/2.
\end{cases}
\end{eqnarray}
\item \label{item:expansionf} $F_\varepsilon\in \dot L^m (\overline {B_\varepsilon(p_0)}) \cap C^\infty(B_\varepsilon(p_0)) $ for all $m\in (1, 1/\alpha)$ and in the coordinate system \eqref{rescaled geod coord}, $F_\varepsilon$ has expansion
{
\begin{equation}
\label{expansion fepsilon}
F_\varepsilon(\psi_\varepsilon(x)) = - \frac{|M|}{\varepsilon^n}\left(\int_{\B^n} \dfrac{\dd x}{(1-|x|^2)^\alpha}\right)^{-1} \left( \dfrac{1}{(1-|x|^2)^{\alpha}}+\mc{O}_{\dot{L}^m(\overline{\B^n})}(E(\alpha,\varepsilon))\right).
\end{equation}
}
\end{enumerate}
\end{prop}

\begin{rmk}\label{remark torus}
If $(M, g)$ is $\T^n$, then the error term \eqref{Ealphaeps} can be replaced by 
\eqref{eq:newerrorterm}.
\end{rmk}

Although the statements may sound natural, the proof of Propositions \ref{proposition uepsilon} and \ref{proposition fepsilon} is actually involved, due to the nonlocality of the generator $\mathscr{A}$ which causes trouble understanding the analytic properties of $u_\varepsilon, F_\varepsilon$ on $\partial B_\varepsilon(p_0)$. We will mainly follow the strategy of \cite{getoor1961first} which deals with a similar problem in a simpler setting where $\mathscr{A}$ is the fractional Laplacian in $\R^n$. Since they are technical, the proofs are deferred to \S\ref{sec:technical} below.

\subsection{Proof of the stopping time Theorems}

We first show how Theorems \ref{narrow capture}, \ref{narrow capture torus} and \ref{narrow capture sphere} can be deduced from Propositions \ref{proposition uepsilon} and \ref{proposition fepsilon}.

\subsubsection{Proof of Theorems \ref{narrow capture} and \ref{narrow capture torus}}

We start with the Anosov case.

\begin{proof}[Proof of Theorem \ref{narrow capture}] First of all, observe that, using Proposition \ref{proposition fepsilon}, the following holds: let $G \in \D'(M)$ be smooth in a neighourhood of $p_0$, then for all $\varepsilon>0$ sufficiently small, 
\begin{equation}
\label{exp}
\int_{B_\varepsilon(p_0)} F_\varepsilon(q) G(q){\rm dvol}_g(q) =  |M| G(p_0) + \mc{O}(E(\alpha,\varepsilon)),
\end{equation}
where $E(\alpha,\varepsilon)$ is given by \eqref{Ealphaeps}.

We now fix $p\neq p_0$ assume that $\varepsilon>0$ is sufficiently small such that $p\notin B_\varepsilon(p_0)$. Since $u_\varepsilon - C_\varepsilon = \mathscr{A}^+(F_\varepsilon - \mathbf{1}_{\Omega_{\varepsilon}})$, we have:
\begin{eqnarray*}
u_\varepsilon(p) - C_\varepsilon  &=& \left(\mathscr{A}^+ \1_{B_\varepsilon(p_0)}\right)(p) + \int_{B_\varepsilon(p_0)} \mathrm{G}_{\mathscr{A}}(p,q)F_\varepsilon(q) {\rm dvol}_g(q).\\
&=&O(\varepsilon^n) + \int_{B_\varepsilon(p_0)} \mathrm{G}_{\mathscr{A}}(p,q)F_\varepsilon(q) {\rm dvol}_g(q).
\end{eqnarray*}

Since $\mathscr{A}^+$ is a pseudodifferential operator, $ \mathrm{G}_{\mathscr{A}}(p,q)$ is smooth away from the the set $\{p = q\}$. Since we have taken $\varepsilon>0$ so that $p\notin B_\varepsilon(p_0)$, $F_\varepsilon(q)$ is integrated against a smooth function of $q$. We now use the expansion produced in \eqref{exp} to get
$$u_\varepsilon(p) - C_\varepsilon = |M| \mathrm{G}_{\mathscr{A}}(p,p_0) +\mc{O}(E(\alpha,\varepsilon)).$$
Recalling now that $\Omega_\varepsilon = M\setminus B_\varepsilon(p_0)$ and $C_\varepsilon = |M|^{-1} \int_M u_\varepsilon(q) {\rm dvol}_g(q)$ has expansion given in Proposition \ref{proposition fepsilon} concludes the proof of Theorem \ref{narrow capture}.
\end{proof}
We now pass to the torus case.
\begin{proof}[Proof of Theorem \ref{narrow capture torus}]
We may redo exactly the proof of Theorem \ref{narrow capture}, and taking into account Remark \ref{remark torus} yields the right error term.
\end{proof}

\subsubsection{Proof of Theorem \ref{narrow capture sphere}} We now deal with the case of the sphere. We will need the preliminary lemma:

\begin{lem}
\label{distance in rescaled coordinates}
We have that for any $\mu >0$,
$$\frac{1}{\mathrm{dist}_g(\psi_\varepsilon(x),\psi_\varepsilon(y))^{n-\mu}}  = \frac{1}{\varepsilon^{n-\mu}|x-y|^{n-\mu}}  +\varepsilon^{2} \frac{A\left(\varepsilon, x, |x-y|,\frac{x-y}{|x-y|}\right)}{\varepsilon^{n-\mu} |x-y|^{n-\mu} }$$
for some smooth function $A: [0,\varepsilon_0) \times \B\times \R\times S^{n-1} \to \R$.
\end{lem}
\begin{proof}
This is a direct consequence of \cite[Corollary 2.3]{nursultanov2021mean}.
\end{proof}

We can now proceed with the proof of Theorem \ref{narrow capture sphere}.

\begin{proof}[Proof of Theorem \ref{narrow capture sphere}]
For $p\notin \{p_0, -p_0\}$ computing $u_\varepsilon (p) - C_\varepsilon$ goes verbatim as in proof of Theorem \ref{narrow capture} by choosing $\varepsilon>0$ small enough so that $p\notin B_\varepsilon(p_0)\cup B_\varepsilon (-p_0)$. Therefore we will not repeat it here. The interesting part is computing $u_\varepsilon(-p_0) - C_\varepsilon$. To do so we first recall that by Lemma \ref{microlocal structure of A+ sphere}, 
$$\mathscr{A}^+ = \mathscr{K} + \mathscr{B}\mathscr{J} + \mathscr{S}_1 + \mathscr{S}_2\mathscr{J}$$
where $\mathscr{K}$ and $\mathscr{B}$ are described in Lemma \ref{parametrix of A sphere} and $\mathscr{S}_1, \mathscr{S}_2$ are pseudodifferential operators which are more regular.

\begin{lem}
\label{subprincipal part}
The operator $\mathscr{B}$ is of the form
$$\mathscr{B} = \mathscr{L} ' + \mathscr{R}'$$
where $\mathscr{L}'$ has Schwartz kernel 
$
-\chi(\mathrm{dist}_g(p,q)^2)\mathrm{dist}_g(p,q)^{-n+1 + 4\alpha}
$
and $\mathscr{R}'$ is the sum of classical pseudodifferential operators of order at most $\max(-4\alpha -2, -6\alpha -1)$. Furthermore, as $\varepsilon \to 0$,
\begin{equation}
\label{L' in coord}
(\mathscr{L}' \mathscr{J} F_\varepsilon)(-p_0) =\frac{-|M|}{\varepsilon^{n-1-4\alpha} }\left( \int_{\B^n} \frac{\dd x}{(1-|x|^2)^\alpha}\right)^{-1} \int_{0}^1 \frac{r^{4\alpha}}{(1-r^2)^{\alpha}}dr + o(\varepsilon^{-n + 1+4\alpha}).
\end{equation}
\end{lem}
\begin{proof}
The expression for $\mathscr{B}$ is a consequence of the fact that $\mathscr{L}'$ is a classical pseudodifferential operator with principal symbol $-|\eta|^{-1-4\alpha}_g$ which is the same as the principal symbol of $\mathscr{B}$ as given by Lemma \ref{parametrix of A sphere}.

To obtain \eqref{L' in coord}, use a rescaled geodesic coordinate system centred at $-p_0$ given by
$$\tilde \psi_\varepsilon(x) = \exp_{-p_0}(\varepsilon x_1E_1+ \dots + \varepsilon x_nE_n)$$
for $E_1,\dots, E_n\in T_{-p_0}\mS^n$ an orthonormal set of vectors over $-p_0$. Using Lemma \ref{distance in rescaled coordinates} we see that the coordinate expression for $\displaystyle{-\frac{\chi(\mathrm{dist}_g(p,q)^2) }{\mathrm{dist}_g(p,q)^{n-1 - 4\alpha}}}$ in these coordinates is:
$$\frac{1}{\mathrm{dist}_g(\tilde \psi_\varepsilon(x),\tilde\psi_\varepsilon(y))^{n-1-4\alpha}}  = \frac{1}{\varepsilon^{n-1-4\alpha}|x-y|^{n-1-4\alpha}}  +\varepsilon^{2} \frac{A\left(\varepsilon, x, |x-y|,\frac{x-y}{|x-y|}\right)}{\varepsilon^{n-1-4\alpha} |x-y|^{n-1-4\alpha} }$$
for some smooth function $A: [0,\varepsilon_0) \times \B\times \R\times S^{n-1} \to \R$.
So writing out $\left(\mathscr{L}'\mathscr{J} F_\varepsilon \right)(-p_0)$ in these coordinates and use the expression 
$$\psi_\varepsilon^* {\rm dvol}_g = \varepsilon^n(1 + \varepsilon^2 Q_\varepsilon(x)) \dd x$$
for the volume form with $Q_\varepsilon(x)$ smooth and uniformly bounded in $\varepsilon$, we get
\begin{eqnarray*}
\left(\mathscr{L}'\mathscr{J} F_\varepsilon \right)(-p_0) &=&- \varepsilon^{1+4\alpha}\int_{\B^n}\frac{(\mathscr{J}F_\varepsilon)(\tilde \psi_\varepsilon(x))}{|x|^{n-1-4\alpha}} (1+\varepsilon^2Q_\varepsilon(x))\dd x\\ &-&\varepsilon^{3+4\alpha}\int_{\B^n} (\mathscr{J}F_\varepsilon) (\tilde \psi_\varepsilon(x))\frac{A\left(\varepsilon, x, |x|,\frac{x}{|x|}\right)}{ |x|^{n-1-4\alpha} }(1+\varepsilon^2Q_\varepsilon(x))\dd x.
\end{eqnarray*}Clearly $(\mathscr{J} F_\varepsilon)(\tilde \psi_\varepsilon(x)) = F_\varepsilon(\psi_\varepsilon(x))$ where $\psi_\varepsilon$ is the rescaled geodesic coordinate \eqref{rescaled geod coord} centred at $p_0$. So we have
\begin{eqnarray*}
\left(\mathscr{L}'\mathscr{J} F_\varepsilon \right)(-p_0) &=& -\varepsilon^{1+4\alpha}\int_{\B^n}\frac{F_\varepsilon( \psi_\varepsilon(x))}{|x|^{n-1-4\alpha}} (1+\varepsilon^2Q_\varepsilon(x))\dd x\\ &-&\varepsilon^{3+4\alpha}\int_{\B^n} F_\varepsilon( \psi_\varepsilon(x))\frac{A\left(\varepsilon, x, |x|,\frac{x}{|x|}\right)}{ |x|^{n-1-4\alpha} }(1+\varepsilon^2Q_\varepsilon(x))\dd x.
\end{eqnarray*}
 Now, plugging the expansion of $F_\varepsilon(\psi_\varepsilon(x))$ given by \emph{\ref{item:expansionf}} in Proposition \ref{proposition fepsilon} into the above integral, we get \begin{eqnarray*}
\varepsilon^{1+4\alpha}\int_{\B^n} \frac{F_\varepsilon(\psi_\varepsilon(x))}{|x|^{n-1-4\alpha}}\dd x  =-\varepsilon^{1+4\alpha-n}|M|\tilde c(n,\alpha)
- \int_{\B^n} \frac{O_{\dot L^m(\overline\B)}\left(\varepsilon^{1-n+4\alpha} E(\alpha,\varepsilon)\right)}{|x|^{n-1-4\alpha}}\dd x,
\end{eqnarray*}
where 
$$
\tilde c(n, \alpha) = \left(\int_{\B^n} \frac{\dd x}{(1-|x|^2)^\alpha}\right)^{-1}.
$$
By our assumption that $1>(n-4)\alpha$, we can choose $m\in (1,1/\alpha)$ so that the function $x \mapsto |x|^{-n+1+4\alpha}$ lies in $L^{m'}(\B)$, where $1/m + 1/m' = 1$. Now by H\"older inequality we get \eqref{L' in coord}, which concludes the proof.
\end{proof}

Inserting the expression for $\mathscr{A}^+$ obtained in Lemma \ref{subprincipal part} and Lemma \ref{microlocal structure of A+ sphere} into the equation $u_\varepsilon  - C_\varepsilon= \mathscr{A}^+\left(F_\varepsilon - \1_{\Omega_\varepsilon}\right)$, we see that 
\begin{eqnarray}
\label{sphere integral eq decomposed}\nonumber
u_\varepsilon(-p_0) -C_\varepsilon &=&\left(\mathscr{A}^+ \1_{B_\varepsilon}\right)(-p_0) + \left(\mathscr{L} F_\varepsilon\right)(-p_0) +\left(\mathscr{R} F_\varepsilon\right)(-p_0) \\&+& \left( \mathscr{L}'\mathscr{J} F_\varepsilon\right)(-p_0) +\left( \mathscr{R}'\mathscr{J} F_\varepsilon\right)(-p_0)
\end{eqnarray}
where $\mathscr{L}$ and $\mathscr{R}$ are pseudodifferential operators and $\mathscr{R}'$ is the sum of classical pseudodifferential operators of order at most $\max(-4\alpha -2, -6\alpha -1)$.

The Schwartz kernel of $\mathscr{L}$ is $\chi(\mathrm{dist}_g(p,q)^2)\mathrm{dist}_g(p,q)^{-n+2\alpha}$ (see Corollary \ref{A+ = L + R sphere}) which vanishes for $(p,q)$ near $(-p_0,p_0)$. Since $\mathscr{R}$ is a pseudodifferential operator, its Schwartz kernel is smooth for $(p,q)$ near $(-p_0,p_0)$. By Lemma \ref{subprincipal part} the operator $\mathscr{R}'$ has a weaker singularity than $\mathscr{L}'$. So \eqref{sphere integral eq decomposed} becomes
\begin{eqnarray*}
u_\varepsilon(-p_0) - C_\varepsilon &=& \left( \mathscr{L}'\mathscr{J} F_\varepsilon\right)(-p_0)  + o(\varepsilon^{-n+1+4\alpha})
\end{eqnarray*}
Now insert the formula \eqref{L' in coord} for $\left( \mathscr{L}'\mathscr{J} F_\varepsilon\right)(-p_0)$ into the above identity we get
{\begin{eqnarray}
\nonumber
u_\varepsilon(-p_0) - |M|^{-1}\int_M u_\varepsilon{\rm dvol}_g=\frac{-|M|\tilde c(n,\alpha)}{\varepsilon^{n-1-4\alpha} }
+ o(\varepsilon^{-n+1+4\alpha})
\end{eqnarray}}
which is the identity \eqref{sphere blow up}. This completes the proof of Theorem \ref{narrow capture sphere}.
\end{proof}

\section{Proof of the fundamental properties of the expected stopping time}

\label{sec:technical} 

The aim of this section is to prove Propositions \ref{proposition uepsilon} and \ref{proposition fepsilon}.

\subsection{Structure of the argument} Although stated this way, Propositions \ref{proposition uepsilon} and \ref{proposition fepsilon} do not reflect the structure of the argument which is quite involved due to the technical issues caused by the nonlocality of the generator $\mathscr{A}$. Basically, the main issue is that we cannot show \emph{directly} that $\mathscr{A} u_\varepsilon = -1$ on $\Omega_\varepsilon$. The idea, somehow, is to revert the logic of the argument. First, let us make some quick observations. Assuming that the fundamental equation $\mathscr{A} u_\varepsilon = -1$ holds on $\Omega_\varepsilon = M \setminus \overline{B_\varepsilon(p_0)}$ and $u_\varepsilon = 0$ on $\overline{B_\varepsilon(p_0)}$, that is, \eqref{equation defF} holds, we get by applying $\mathscr{A}^+$ to both sides of \eqref{equation defF} that
\[
\mathscr{A}^+ \mathscr{A} u_\varepsilon = u_\varepsilon - |M|^{-1} \int_M u_\varepsilon(p) {\rm dvol}_g(p) = \mathscr{A}^+(F_\varepsilon-\mathbf{1}_{\Omega_{\varepsilon}}),
\]
that is $\mathscr{A}^+ (F_\varepsilon - \mathbf{1}_{\Omega_{\varepsilon}})  = u_\varepsilon - C_\varepsilon$. Since $u_\varepsilon$ vanishes on $B_\varepsilon(p_0)$, we thus get:
\begin{equation}
\label{eq for Fe1'}
\mathscr{A}^+ (F_\varepsilon-\mathbf{1}_{\Omega_{\varepsilon}}) =  - C_\varepsilon , \qquad \text{ on } B_\varepsilon(p_0).
\end{equation}
Moreover, integrating \eqref{equation defF} with respect to ${\rm dvol}_g$, we get
\begin{equation}
\label{compatibility'}
\langle F_\varepsilon, {\rm dvol}_g \rangle = |\Omega_\varepsilon|.
\end{equation}
The pair of equations \eqref{eq for Fe1'} - \eqref{compatibility'} with unknowns $(F_\varepsilon, C_\varepsilon)$ will be called the \emph{integral equation}. The argument then goes as follows:   \\

\begin{enumerate}
\item \textbf{Existence and uniqueness to the integral equation.}  First, we \emph{construct} a pair of solution $(\widetilde{F}_\varepsilon, \widetilde{C}_\varepsilon)$ to the integral equation \eqref{eq for Fe1'} - \eqref{compatibility'} such that $\widetilde{F}_\varepsilon$ has support in $\overline{B_\varepsilon(p_0)}$ and control precisely its analytic properties, that is, show that it satisfies the content of Proposition \ref{proposition fepsilon}. More precisely, we will show:

\begin{prop}[Existence and uniqueness of regular solutions to the integral equation]
\label{Feps asymptotic}Let $(M,g)$ be either Anosov, the flat torus, or the round sphere. For $\varepsilon>0$ small enough, there \emph{exists} a \emph{unique} $\widetilde{F}_\varepsilon\in \dot L^m (\overline {B_\varepsilon(p_0)}) \cap C^\infty(B_\varepsilon(p_0)) $ with $m\in (1, 1/\alpha)$ and constant $\widetilde{C}_\varepsilon$ solving \eqref{eq for Fe1'} - \eqref{compatibility'}. Moreover, $\widetilde{C}_\varepsilon$ satisfies the expansion \eqref{expansion cepsilon} and $\widetilde{F}_\varepsilon$ satisfies the expansion \eqref{expansion fepsilon} in Proposition \ref{proposition fepsilon}.
\end{prop}

Recall here that the spaces $\dot L^m$ were introduced in \S\ref{sssection:functional-spaces}. The proof of Proposition \ref{Feps asymptotic} is the content of \S\ref{ssection:existence}. The proof of Proposition \ref{Feps asymptotic} relies on local argument involving Fourier analysis. \\

\item \textbf{Uniqueness of solutions to the fundamental equation.} We then set
\begin{equation}
\label{equation utilde-epsilon}
\widetilde{u}_\varepsilon := \mathscr{A}^+ (\widetilde{F}_\varepsilon-  \mathbf{1}_{\Omega_{\varepsilon}}) + \widetilde{C}_\varepsilon .
\end{equation}
By construction, the distribution $\widetilde{u}_\varepsilon$ satisfies $\mathscr{A} \widetilde{u}_\varepsilon = - 1$ on $\Omega_\varepsilon$, $\widetilde{u}_\varepsilon = 0$ on $B_\varepsilon(p_0)$. Moreover, as a consequence of Proposition \ref{Feps asymptotic}, $\widetilde{u}_{\varepsilon}$ lies in some Sobolev space of positive regularity, that is, $\widetilde{u}_\varepsilon \in \dot{W}^{2\alpha, m}(\overline{\Omega_\varepsilon})$. We will show that the following uniqueness result holds:

\begin{prop}[Uniqueness of regular solutions to the fundamental equation]
\label{proposition uniqueness}
Let $(M,g)$ be the round sphere, the torus, or Anosov. Let $w \in \dot{W}^{2\alpha, m}(\overline{\Omega_\varepsilon})$, $1<m<\infty$. Assume that $\mathscr{A} w = - 1$ on $\Omega_\varepsilon$. Then $w = u_\varepsilon$.
\end{prop}

This uniqueness result should be compared with \cite[Corollary 5.1]{getoor1961first}. It will be the content of \S\ref{ssection:uniqueness}. Proposition \ref{proposition uniqueness} will therefore imply that $u_\varepsilon = \widetilde{u}_\varepsilon \in \dot{W}^{2\alpha, m}(\overline{\Omega_\varepsilon})$ satisfies the fundamental equation $\mathscr{A} u_\varepsilon = -1$ on $\Omega_\varepsilon$. The idea behind Proposition \ref{proposition uniqueness} is to use the integral representation \eqref{eq:integral-representation}, relating $u_\varepsilon$ to the generator of a bounded semi-group on $\Omega_\varepsilon$. 

\end{enumerate}

The proofs of Propositions \ref{proposition uepsilon} and \ref{proposition fepsilon} are then straightforward, combining both Proposition \ref{Feps asymptotic} and \ref{proposition uniqueness}.

\begin{proof}[Proof of Proposition \ref{proposition uepsilon}]
By the previous Propositions, we have that $u_\varepsilon = \widetilde{u}_\varepsilon \in \dot{W}^{2\alpha,m}(\overline{\Omega_\varepsilon})$ and $u_\varepsilon$ satisfies the fundamental relation \eqref{equation uepsilon}. Moreover, in the case of Anosov manifolds or the torus, $u_\varepsilon \in C^\infty(M \setminus \partial B_\varepsilon(p_0))$ follows from standard elliptic regularity since $u_\varepsilon = 0$ in $B_\varepsilon(p_0)$, $\mathscr{A} u_\varepsilon = -1$ on $\Omega_{\varepsilon}$ and $\mathscr{A}$ is pseudodifferential elliptic. In the sphere case, we get similarly that $u_\varepsilon \in C^\infty(M \setminus \partial B_\varepsilon(\pm p_0))$ by elliptic regularity of $\mathscr{A}$ (up to antipodal points). Eventually, the proof that $u_\varepsilon \in L^\infty(M)$ is deferred to Corollary \ref{ueps is L infinity} below and will be a consequence of the representation formula \eqref{eq:integral-representation} for $u_\varepsilon$.
\end{proof}

\begin{proof}[Proof of Proposition \ref{proposition fepsilon}]
It suffices to prove that $\widetilde{F}_\varepsilon = F_\varepsilon$, where $F_\varepsilon$ is defined by \eqref{equation defF} and $\widetilde{F}_\varepsilon$ solves the integral equation \eqref{eq for Fe1'} - \eqref{compatibility'}, and similarly that $\widetilde{C}_\varepsilon = C_\varepsilon$. But by definition of $\widetilde{u}_\varepsilon$ in \eqref{equation utilde-epsilon}, and by \eqref{equation defF}, one has using Proposition \ref{proposition uniqueness}:
\[
\widetilde{u}_\varepsilon = \mathscr{A}^+(\widetilde{F}_\varepsilon-\mathbf{1}_{\Omega_\varepsilon}) + \widetilde{C}_\varepsilon = u_\varepsilon = \mathscr{A}^+(F_\varepsilon-\mathbf{1}_{\Omega_\varepsilon}) + C_\varepsilon.
\]
Integrating the previous equation over $M$ yields $C_\varepsilon = \widetilde{C}_\varepsilon$. Hitting the previous equation with $\mathscr{A}$ then yields $F_\varepsilon = \widetilde{F}_\varepsilon$. This concludes the proof.
\end{proof}

\subsection{Existence of solutions to the integral equation}

\label{ssection:existence}

The aim of this paragraph is to construct a solution to \eqref{eq for Fe1'} and \eqref{compatibility'}, that is, to prove Proposition \ref{Feps asymptotic}. The first step is to rewrite a bit \eqref{eq for Fe1'} more explicitly. For that, fix $p_0\in M$ and let $\varepsilon\in(0, \frac{r_{\mathrm{inj}}}{10})$ be small. When $(M,g)$ is the torus, Anosov, or the sphere, by Theorem \ref{torus} (torus), Lemma \ref{singular part of A+} (Anosov), and Corollary \ref{A+ = L + R sphere} (sphere), $\mathscr{A}^+$ can be written as 
$$\mathscr{A}^+ = \mathscr{L} + \mathscr{R}$$ 
where $\mathscr{R} : C^\infty(M)\to \D'(M)$ is such that the map
$$ u \mapsto (\mathscr{R} u)|_{B_\varepsilon(p_0)}$$
with domain $\dot W^{s,m}(\overline {B_\varepsilon(p_0)})$ can be represented by a finite sum of classical pseudodifferential operators each of which is of order at most $\max(-4\alpha, -2\alpha-1)$. In the case of the torus, since $\mathscr{A}$ is just the fractional Laplacian, $\mathscr{R}\in \Psi^{-1-2\alpha}_{cl}(M)$.

As a consequence, using this decomposition and that
\[
\mathscr{A}^+ \mathbf{1} = 0 = \mathscr{A}^+ (\mathbf{1}_{B_{\varepsilon}(p_0)} + \mathbf{1}_{\Omega_{\varepsilon}}),
\]
we get
\begin{equation}
\label{equation:feps-rewritten}
(\mathscr{L} + \mathscr{R}) (F_\varepsilon + \mathbf{1}_{B_{\varepsilon}(p_0)} ) =  - C_\varepsilon , \qquad \text{ on } B_\varepsilon(p_0).
\end{equation}
Observe that both functions $F_{\varepsilon}$ and $\mathbf{1}_{B_{\varepsilon}(p_0)}$ are supported on the small ball $\overline{B_\varepsilon(p_0)}$. The first step is to study the boundedness properties of $\mathscr{R}$ on functions supported on $\mathbf{1}_{B_{\varepsilon}(p_0)}$.

\subsubsection{Boundedness properties of $\mathscr{R}$ on small balls} We characterize the mapping property of $\mathscr{R}$ on the small ball $B_\varepsilon(p_0)$ of radius $\varepsilon > 0$. We will need the following preliminary standard result:

\begin{lem}
\label{uniform operator bound}
For $\varepsilon \geqslant 0$, let $A_\varepsilon \in C^\infty_c( \R^n_x\times \R_r\times S^{n-1}_\omega)$ be a family of compactly supported functions, with support uniformly bounded in $\varepsilon \in [0,\varepsilon_0]$ and whose derivatives are also uniformly bounded in $\varepsilon$. Then, the family of operators $T_\varepsilon$ given by
$$T_\varepsilon f (x):= \int\limits_{\omega\in S^{n-1}}\int_0^\infty A_\varepsilon(x,r,\omega) r^l f(x+ r\omega) dr d\omega$$
is a family of pseudodifferential operators of order $-1-l$, with uniform boundedness properties in terms of $\varepsilon$. In particular,
\[
T_\varepsilon : W^{s,m}(\R^n) \to W^{s+l+1,m}(\R^n)
\]
is bounded for all $s \in \R, m \in (1,\infty)$, with uniform bounds in $\varepsilon \geqslant 0$.
\end{lem}

\begin{proof}
Observe that for all $(x,\xi) \in T^*\R^n$,
\[
T_\varepsilon(e^{i x \cdot \xi}) = e^{i x \cdot \xi} \sigma_{\varepsilon}(x,\xi),
\]
for
\[
\sigma_\varepsilon(x,\xi) := \int_{\omega \in S^{n-1}} \int_0^{+\infty} A_{\varepsilon}(x,r,\omega) r^l e^{i r \omega \cdot \xi} dr d\omega,
\]
which is compactly supported in the $x$-variable. It is immediate to check that the following estimates hold: for all $\beta \in \N^n, \gamma \in \N^n$, there exists $C := C(\beta,\gamma)$ such that
\[
|\partial^\beta_x \partial^\gamma_\xi \sigma_{\varepsilon}(x,\xi)|\leqslant C \langle \xi \rangle^{-1-l-|\gamma|}, \qquad \forall (x,\xi) \in T^*\R^n,  \quad \forall \varepsilon \geqslant 0.
\]
This proves that $(T_{\varepsilon})_{\varepsilon \geqslant 0}$ is indeed a family of pseudodifferential operators with uniform estimates, see \cite[Theorem 3.4]{Grigis-Sjostrand-94}.
\end{proof}

The previous lemma has the following consequence for the boundedness of $\mathscr{R}$ on small balls:

\begin{lem}
\label{coordinate bound for remainder}
Let $R_\varepsilon: C^\infty_c(\B) \to \D'(\B)$ be the coordinate representation of $\mathscr{R}$ in the coordinate system \eqref{rescaled geod coord}, that is, $R_\varepsilon f := \psi_\varepsilon^*  \mathscr{R}  (\psi_\varepsilon^{-1})^*$. Then we have the following estimate
\begin{eqnarray}
\label{R bound}
\|R_\varepsilon\|_{{\dot L}^{m}(\overline \B) \to {\overline W}^{\min(2\alpha+1, 4\alpha),m}(\B)} \leqslant C \varepsilon^{2\alpha} E(\alpha, \varepsilon)
\end{eqnarray}
where $E(\alpha, \varepsilon)$ is defined as \eqref{Ealphaeps}, and $C > 0$ is some positive constant.
\end{lem}

\begin{proof}
By assumption the operator $u\mapsto (\mathscr{R} u)\mid_{B_\varepsilon(p_0)}$ with domain $\dot W^{s,m}(\overline{B_\varepsilon(p_0)})$ can be represented by an element of $\Psi^{-2\alpha-1}_{\mathrm{cl}}(M) + \Psi^{-4\alpha}_{\mathrm{cl}}(M)$. So, by \cite[Page 285, (8.45)]{taylor1996partial}, in the coordinate system given by \eqref{rescaled geod coord}, the Schwartz kernel of $R_{\varepsilon}$ is given, if $\alpha \neq 1/2$, by
\[
 \varepsilon^{2\alpha+1}\frac{ A_\varepsilon\left( x, |x-y|, \frac{x-y}{|x-y|}\right) }{ |x-y|^{n-2\alpha-1}}+ \varepsilon^{4\alpha}\frac{ B_\varepsilon\left( x,  |x-y|,\frac{x-y}{|x-y|}\right)}{ |x-y|^{n-4\alpha}} + \varepsilon^n \kappa(\varepsilon x),
 \]
 and if $\alpha = 1/2$ by
 \[
\varepsilon^2 \frac{A_\varepsilon\left( x,  |x-y|,\frac{x-y}{|x-y|}\right)}{|x-y|^{n-2}} + \varepsilon^{2} P(\varepsilon x,\varepsilon (x-y)),
\]
where $A_\varepsilon(x,r,\omega)$ and $B_\varepsilon(x,r,\omega)$ satisfy the hypothesis of Lemma \ref{uniform operator bound} and $\kappa$ is a smooth function. The functions $P(x,z)$ satisfy $P(x,z) \sim \sum_{j\geqslant 0} P_j(x,z)\log |z|$ with each $P_j(x,z) \in C^\infty({\B^n} \times \R^n)$ a homogeneous polynomial of degree $j$ in the $z$ variable. The notation $P(x,z) \sim \sum_{j\geqslant 0} P_j(x,z)\log|z|$ means that for all $k\in \N$ there exists $N\in \N$ such that
$$P(x,z) - \sum_{j=0}^NP_j(x,z)\log |z| \in C^k(\B\times\R^n).$$
We can now apply Lemma \ref{uniform operator bound} to obtain the desired estimate.
\end{proof}

\subsubsection{Solving the integral equation}

Our aim is to rewrite more explicitly \eqref{equation:feps-rewritten}. For that purpose, we introduce the operator
\begin{equation}
\label{equation lalpha}
L_\alpha u :=- \int_{\B^n} \frac{u(y)}{|x-y|^{n-2\alpha}} dy,
\end{equation}
defined for $u \in C^\infty(\overline{\B^n})$. Using the rescaled normal coordinates \eqref{rescaled geod coord}, we set 
\begin{eqnarray}
\label{tilde F def}
\tilde F_\varepsilon(x) := F_\varepsilon (\psi_\varepsilon (x)).
\end{eqnarray}
The following holds:

\begin{lem}
\label{lemma ftilde epsilon}
Written in the geodesic normal coordinates \eqref{rescaled geod coord}, the equation \eqref{equation:feps-rewritten} is equivalent to
\begin{equation}
\label{eq for Fe new}
(L_\alpha + R'_\varepsilon) (\tilde{F}_\varepsilon + \mathbf{1}_{\B^n}) = - \dfrac{C_\varepsilon \varepsilon^{-2\alpha}}{C(n,-\alpha)} \mathbf{1}_{\B^n}, \quad \text{ on } \B,
\end{equation}
where the operator $R'_\varepsilon : \dot L^m(\B) \to \overline{W}^{\min(2\alpha+1, 4\alpha),m}(\B)$ satisfies the bound:
\begin{equation}
\label{equation bound r}
\|R'_\varepsilon\|_{\dot L^m \to \overline{W}^{\min(2\alpha+1, 4\alpha),m}} \leqslant C E(\alpha,\varepsilon).
\end{equation}
\end{lem}

\begin{proof}
Using the decomposition $\mathscr{A}^+ = \mathscr{L} + \mathscr{R}$, we can rewrite \eqref{equation:feps-rewritten} as
\begin{equation}
\label{eq for Fe}
\begin{split}
-C(n,-\alpha)\int_{B_\varepsilon} \frac{1 + F_\varepsilon(q)}{   d_g(p,q)^{n-2\alpha}}& {\rm dvol}_g(q)+ \mathscr{R} (\1_{B_\varepsilon} + F_\varepsilon)  = -C_\varepsilon.
 \end{split}
\end{equation}
Observe that in the geodesic normal coordinates \eqref{rescaled geod coord}, 
\begin{eqnarray}
\label{volume form in coordinates}
\psi_\varepsilon^*{\rm dvol}_g(y) = \varepsilon^n (1+ \varepsilon^2 Q_\varepsilon(y)) \dd y
\end{eqnarray}
for some smooth function $Q_\varepsilon$ whose derivatives are also uniformly bounded in $\varepsilon>0$. The compatibility condition \eqref{compatibility'} becomes
\begin{eqnarray}
\label{coordinate compatibility}
\varepsilon^n\int_{\B^n} \tilde F_\varepsilon(x) (1+\varepsilon^2Q_\varepsilon(x))\dd x = |\Omega_\varepsilon|
\end{eqnarray}
In the geodesic normal coordinates, using Lemma \ref{distance in rescaled coordinates} and the above formulas for the change of volume, the terms of \eqref{eq for Fe} have the following expression:
\[
-\int_{B_\varepsilon} \frac{1}{   d_g(p,q)^{n-2\alpha}}  {\rm dvol}_g(q)  = \varepsilon^{2\alpha} L_\alpha \mathbf{1}_{\B^n} +  \eps^{2+2\alpha}  K_\varepsilon \mathbf{1}_{\B^n},
\]
and
\[
-\int_{B_\varepsilon}\frac{F_\varepsilon(q)}{   d_g(p,q)^{n-2\alpha}}  {\rm dvol}_g(q) = \varepsilon^{2\alpha} L_\alpha \tilde{F}_\varepsilon + \varepsilon^{2+2\alpha} K_\varepsilon \tilde{F}_\varepsilon,
\]
where $K_\varepsilon$ is the operator on the unit ball $\B$ defined by:
\[
K_\varepsilon u (x) := - \int_{\B^n} \dfrac{Q_\varepsilon(y) + A_{\varepsilon}(x,|x-y|,x-y/|x-y|)(1+\varepsilon^2Q_\varepsilon(y))}{|x-y|^{n-2\alpha}} u(y) \dd y.
\]
We then set
\[
R'_\varepsilon := \varepsilon^2 K_\varepsilon + \dfrac{\varepsilon^{-2\alpha} R_\varepsilon}{C(n,-\alpha)},
\]
so that equality \eqref{eq for Fe new} is satisfied.

It remains to show that $R'_\varepsilon$ satisfies the stated bound. The term involving $R_\varepsilon$ yields the bound \eqref{equation bound r} by Lemma \ref{coordinate bound for remainder}. As to the term involving $K_\varepsilon$, using Lemma \ref{uniform operator bound}, is is easily seen to be of size $\mc{O}(\varepsilon^{2+2\alpha})$ as a bounded operator $\dot L^m(\B) \to \overline{W}^{\min(2\alpha+1, 4\alpha),m}(\B)$, and this is in turn a $o(E(\alpha,\varepsilon))$.
\end{proof}

We now show the following:

\begin{lem}
\label{L inverse}
The operator $L_\alpha : \dot L^m(\overline\B) \to {\overline W}^{2\alpha,m}(\B)$ defined by \eqref{equation lalpha} is a continuous isomorphism for $m\in (1,1/\alpha)$. Hence, for $m\in (1,1/\alpha)$, there exists a bounded operator  $L_\alpha^{-1}: {\overline W}^{2\alpha,m }(\B) \to \dot L^m(\overline \B)$ such that $L^{-1}_\alpha L_\alpha  = L_\alpha L_\alpha^{-1}= \mathbf{1}$.
\end{lem}

\begin{proof}
The operator $L_\alpha$ has Schwartz kernel $|x-y|^{-n+2\alpha}$ and therefore belongs to $\Psi^{-2\alpha}_{\mathrm{cl}}(\R^n)$ with full symbol $-c|\eta|^{-2\alpha}$ for some constant $c>0$ depending on $\alpha$. It therefore satisfies the H\"ormander transmission condition on $\B$ with factorization index $-\alpha$, see \cite[Proposition 1 and Equation (1)]{grubb}. The Fredholm mapping property is then stated in \cite[Theorem 2]{grubb}. 

To see injectivity, suppose $u\in  \dot L^m(\overline \B)$ satisfies $L_\alpha u = 0$. But \cite[Theorem 2.4]{kahane1981solution} (for dimension $2$) and \cite[Theorem 4.3]{kahane1983solution} (for higher dimensions) then asserts that $u=0$.

To see surjectivity, recall that the dual space of $\overline W^{2\alpha, m}(\B)$  is $\dot W^{-2\alpha, m'}(\overline \B)$, see discussion before equation (1.5) and equation (1.8) of \cite{grubb}. So suppose $u\in \dot W^{-2\alpha,m'}(\overline \B)$ is orthogonal to the range
\[
L_\alpha\left( \dot L^m(\overline\B)\right) \subset \overline W^{2\alpha, m}(\B).
\]
Then by the fact that $L_\alpha$ is formally self-adjoint, we have that $L_\alpha u = 0$. The operator $L_\alpha$ is type $-\alpha$ with respect to $\B$, see definition in \cite[Proposition 1 and Equation (1)]{grubb}. Then \cite[Theorem 2]{grubb} forces that $u(x)= (1-|x|^2)^{-\alpha} v(x)$ for some $v\in C^\infty(\overline\B)$. Since $\alpha \in (0,1)$ we have that $u\in L^1(\B)$. Then \cite[Theorem 2.4]{kahane1981solution} (for dimension $2$) and \cite[Theorem 4.3]{kahane1983solution} (for higher dimensions) forces $u=0$. 
\end{proof}

We now complete the proof of Proposition \ref{Feps asymptotic}.

\begin{proof}[Proof of Proposition \ref{Feps asymptotic}] 
First of all, assume that we have a solution $(F_\varepsilon, C_\varepsilon)$ to \eqref{eq for Fe1'} - \eqref{compatibility'} satisfying the properties of Proposition \ref{Feps asymptotic}. By Lemma \ref{lemma ftilde epsilon}, this is the same as having \eqref{eq for Fe new}. Hitting \eqref{eq for Fe new} with $L_\alpha^{-1}$, we then get:
\[
(\mathbf{1}+ L_\alpha^{-1}R'_\varepsilon) (\tilde{F}_\varepsilon + \mathbf{1}_{\B^n}) = - \dfrac{C_\varepsilon \varepsilon^{-2\alpha}}{C(n,-\alpha)} L_\alpha^{-1} \mathbf{1}_{\B^n}
\]
Observe that by Lemmas \ref{lemma ftilde epsilon} and \ref{L inverse}, $L_\alpha^{-1}R'_\varepsilon : \dot{L}^m(\overline{\B^n}) \to\dot{L}^m(\overline{\B^n})$ is a bounded operator with norm $\mc{O}(E(\alpha,\varepsilon)) = o(1)$. As a consequence, we can invert $\mathbf{1}+ L_\alpha^{-1}R'_\varepsilon$ by Neumann series and this yield:
\begin{equation}
\label{equation solution}
\tilde{F}_\varepsilon + \mathbf{1}_{\B^n} = - \dfrac{C_\varepsilon \varepsilon^{-2\alpha}}{C(n,-\alpha)} (\mathbf{1}+ L_\alpha^{-1}R'_\varepsilon)^{-1} L_\alpha^{-1} \mathbf{1}_{\B^n}.
\end{equation}
Integrating with respect to the measure $\varepsilon^n(1+\varepsilon^2Q_\varepsilon(x)) dx$ on $\B$ (rescaled Riemannian measure on the ball, see \eqref{volume form in coordinates}), and using the compatibility condition \eqref{compatibility'}, we get
\[
|M| = - \dfrac{C_\varepsilon \varepsilon^{-2\alpha}}{C(n,-\alpha)} \int_{\B^n} \left((\mathbf{1}+ L_\alpha^{-1}R'_\varepsilon)^{-1} L_\alpha^{-1} \mathbf{1}_{\B^n} \right) (x)  \varepsilon^n(1+\varepsilon^2Q_\varepsilon(x)) dx,
\]
that is
\begin{equation}
\label{equation solution cepsilon}
C_\varepsilon =- \dfrac{\varepsilon^{2\alpha - n}|M| C(n,-\alpha)}{D_\varepsilon},
\end{equation}
where 
\[
D_\varepsilon := \int_{\B^n} \left((\mathbf{1}+ L_\alpha^{-1}R'_\varepsilon)^{-1} L_\alpha^{-1} \mathbf{1}_{\B^n} \right) (x) (1+\varepsilon^2Q_\varepsilon(x)) dx.
\]
Hence, \eqref{equation solution} and \eqref{equation solution cepsilon} show that if such a regular solution $(F_\varepsilon, C_\varepsilon)$ to \eqref{eq for Fe1'} and \eqref{compatibility'} exists, then it must be unique.

Now, we define $(F_{\varepsilon},C_\varepsilon)$ by \eqref{equation solution} and \eqref{equation solution cepsilon}. By construction, they satisfy \eqref{eq for Fe1'} and \eqref{compatibility'} so all that remains to be shown is that they enjoy the asymptotic expansions of Proposition \ref{proposition fepsilon}. We start with $C_\varepsilon$. By \cite[Theorem 3.1]{kahane1981solution} (in dimension 2) and \cite[Theorem 5.1]{kahane1983solution} (in dimension $n\geqslant 3$) we have that 
\begin{eqnarray}
\label{formula for L+1}
(L^{-1}_\alpha \mathbf{1}_{\B^n})(x) =-c_\alpha (1-|x|^2)^{-\alpha},
\end{eqnarray}
where
\[
c_\alpha =
\begin{cases}
\pi^{-2} \sin((1-\alpha)\pi), \quad \text{if } \dim(\B) = 2, \vspace{0.2cm}\\ 
\displaystyle \dfrac{(n-2)\sin((1-\alpha)\pi)\Gamma(n/2-\alpha)}{\pi^{n/2+1}\Gamma(1-\alpha)} \int_0^{1} r^{n-3}(1-r^2)^{-\alpha} ~\mathrm{d}r,\quad \text{if } \dim(\B) \geqslant 3.
\end{cases}
\]
Hence, inserting \eqref{formula for L+1} in the expression of \eqref{equation solution cepsilon}, we easily get that
\begin{equation}
\label{equation:proof-ceps}
C_{\varepsilon} = \varepsilon^{2\alpha - n} |M| C(n,-\alpha) c_\alpha^{-1} \left(\int_{\B^n} \dfrac{\dd x}{(1-|x|^2)^{\alpha}} \right)^{-1}(1+\mc{O}(E(\alpha,\varepsilon))).
\end{equation}
We eventually claim that
the constant $c(n,\alpha)$ introduced in \eqref{equation:cna} is given by
\[
c(n,\alpha) = \frac{C(n,-\alpha)}{c_\alpha} \left(\int_{\B^n} \dfrac{\dd x}{(1-|x|^2)^{-\alpha}} \right)^{-1},
\]
Then, \eqref{equation:proof-ceps} is the content of \eqref{expansion cepsilon}.

Then, inserting the expansion of $C_{\varepsilon}$ into \eqref{equation solution}, using that
\[
(\mathbf{1}+ L_\alpha^{-1}R'_\varepsilon)^{-1} = \mathbf{1} + \mc{O}_{\dot L^m \to \dot L^m}(\|L_\alpha^{-1}R'_\varepsilon\|_{\dot L^m \to \dot L^m}) = \mathbf{1} +  \mc{O}_{\dot L^m \to \dot L^m}(E(\alpha,\varepsilon)),
\]
and the expression \eqref{formula for L+1} for $L^{-1}_\alpha \mathbf{1}_{\B^n}$, we get \eqref{expansion fepsilon}.
\end{proof}

\subsection{Uniqueness of regular solutions to the fundamental equation}

\label{ssection:uniqueness}

The aim of this subsection is to prove Proposition \ref{proposition uniqueness}. For that, we will use the fact that the expected stopping time ${u}_\varepsilon(p) = \mathbb E (\tau_\varepsilon \mid X_0 = p)$ admits a nice integral representation as explained in \S\ref{sec:prelim-markov}, that is, there exists a measurable (in all variables) function $\mathbf{k}$ on $[0,+\infty) \times M \times M$ so that
\begin{equation}
\label{equation integral rep}
u_\varepsilon = \int_0^{+\infty} \int_{\Omega_\eps} \mathbf{k}(t, \cdot, q) \mathrm{dvol}_g(q).
\end{equation}
We claim that the following holds:

\begin{prop}
\label{proposition integral rep}
For all $t>0$, $\mathbf{k}(t,\cdot,\cdot) \in L^\infty(\Omega_{\varepsilon} \times \Omega_{\varepsilon})$. Moreover, if $\mathscr{A}'$ is the generator of the associated semi-group $(T_t)_{t \geqslant 0}$, we have $D(\mathscr{A})\cap L^m(\Omega_\varepsilon) \subset D(\mathscr{A}')\cap L^m(\Omega_\varepsilon)$ for every $1\leqslant  m < \infty$. In particular, for $1<m<\infty$ and $u\in \dot W^{2\alpha,m} (\overline \Omega_\varepsilon)$, one has $\mathscr{A}' u = \1_{\Omega_\varepsilon} \mathscr{A}u$.
\end{prop}

The proof of Proposition \ref{proposition uniqueness} will then be a direct consequence of Proposition \ref{proposition integral rep}, see the end of \S\ref{sssection:generator}.

\subsubsection{Integral representation of the expected stopping time}

\label{ssection:eq uepsilon}

Recall from \S \ref{sec:prelim-markov} that $\mathbf{p}$ is the kernel of the semi-group $(U_t)_{t \geqslant 0}$ introduced in \S\ref{sec:prelim-markov}. We start by observing
\begin{lem}
\label{Linfty estimate for density}
If $(M,g)$ is Anosov, $\T^n$, or $\mathbb{S}^n$, then $\mathbf p(t,\cdot,\cdot) \in C^\infty(M\times M)$ for all $t>0$.
\end{lem}

\begin{proof}
According to \cite[Theorem 1]{Strichartz-functional-calculus-72} applied with $\nu=1$, if $P_1= P$ is an elliptic self-adjoint non-negative pseudo-differential operator and $m_t(x) := e^{- tx}$, we deduce that for any $a>0$ and $N>0$, $e^{-t P}$ is a pseudodifferential operator of order $-N$, uniformly in $t\geqslant a$. This means that the Schwartz kernel $K_P(t,\cdot,\cdot) \in \mc{D}'(M \times M)$ of $e^{-t P}$ satisfies (in every local coordinate patch) the usual estimates for pseudodifferential operators of order $-N < -n$\footnote{For instance, there exists $C > 0$ such that:
\[
\sup_{t \geqslant a, p, q \in M} |K_P(t,p,q)| \leqslant C < \infty.
\]}, and uniformly in time $t \geqslant a$. This implies that the kernel $K_P(t,\cdot,\cdot)$ of $e^{-t P}$ is smooth, uniformly if $t\geqslant a$.

In the case of the torus or Anosov manifold, $-\mathscr{A}$ satisfies the assumptions above. For the sphere, we can still apply the result. Indeed, recall from Theorem \ref{sphere}
\[
\mathscr{A} = \mathscr{A}_{2\alpha} + \mathscr{A}_0+ \mathscr{J}\mathscr{A}_{-1}
\]
where $\mathscr{A}_{2\alpha}\in \Psi^{2\alpha}_{\mathrm{cl}}(M)$, $\mathscr{A}_0\in \Psi^0_{\mathrm{cl}}(M)$, and $\mathscr{A}_{-1}\in \Psi^{-1}_{\mathrm{cl}}(M)$ all commute with $\mathcal{J}$. We use $\mathscr{J}$ to split 
\[
L^2(M) = L^2_{\mathrm{odd}}(M)\oplus L^2_{\mathrm{even}}(M)
\] 
with projections $P_{\mathrm{even}}$ and $P_{\mathrm{odd}}$. The generator $\mathscr{A}$ must preserve this decomposition, and we get 
\[
\mathscr{A} = P_{\mathrm{even}}[ \underset{=\mathscr{A}_{\mathrm{even}}}{\underbrace{\mathscr{A}_{2\alpha}+ \mathscr{A}_0 + \mathscr{A}_{-1}}}] P_{\mathrm{even}}  + P_{\mathrm{odd}}[  \underset{=\mathscr{A}_{\mathrm{odd}}}{\underbrace{\mathscr{A}_{2\alpha}+ \mathscr{A}_0 - \mathscr{A}_{-1}}} ] P_{\mathrm{odd}} .  
\]
Now we can apply \cite[Theorem 1]{Strichartz-functional-calculus-72} to $\mathscr{A}_{\mathrm{even}}$ and $\mathscr{A}_{\mathrm{odd}}$. 
\end{proof}

Since the kernel of $\mathscr{A}$ only contains the constants, we get that as $t\to+\infty$, $\mathbf{p}(t,\cdot,\cdot) \rightharpoonup 1/|M|$ in $\mathcal{D}'(M)$. However, since $t \mapsto \mathbf{p}(t, \cdot, \cdot)$ also uniformly bounded as a map $\left[1, \infty\right[ \to C^\infty(M \times M)$, we deduce that
\begin{equation}\label{eq:limit-p}
\mathbf{p}(t,\cdot,\cdot) \underset{t \to +\infty}{\longrightarrow} 1/|M| \cdot \mathbf{1}_{M\times M},
\end{equation}
where the convergence holds in $C^\infty(M \times M)$.

Recall that $\mathbf{k}$ was introduced in \S\ref{sec:prelim-markov} as the kernel of the semi-group $(T_t)_{t \geqslant 0}$ with generator $\mathscr{A}'$ defined by \eqref{eq:def-T_t}, and that according to \eqref{eq:integral-representation}, $u_\varepsilon$ can be expressed in terms of $\mathbf{k}$. 

\begin{lem}
\label{heat kernel decay}
The generator $\mathscr{A}'$ is negative definite on $L^2(\overline{\Omega_\varepsilon})$. 
\end{lem}

\begin{proof}
We know that $\mathscr{A}'$ has pure point spectrum accumulating at $+\infty$, so it suffices to prove its kernel is trivial. Let thus $\varphi\in \ker \mathscr{A}'$. Then by \eqref{density comparison}, we have:
\[
|\varphi(p)| \leqslant \int_{\Omega_\varepsilon} {\mathbf k}(t,p,q) |\varphi(q)|{\rm dvol}_g(q) \leqslant \int_{\Omega_\varepsilon} {\mathbf p}(t,p,q)|\varphi(q)|{\rm dvol}_g(q).
\]
Take the limit $t\to \infty$ and use \eqref{eq:limit-p} to get for almost-every $p\in \Omega_\varepsilon$:
\[
|\varphi(p)| \leqslant |M|^{-1} \int_{\Omega_\varepsilon} |\varphi(q)|{\rm dvol}_g(q) < |\Omega_{\varepsilon}|^{-1} \int_{\Omega_\varepsilon} |\varphi(q)|{\rm dvol}_g(q),
\]
which easily implies $\varphi=0$. 
\end{proof}
Note that due to Lemma \ref{Linfty estimate for density} and \eqref{density comparison}, we have that $\mathbf k(t,\cdot, \cdot) \in L^\infty(M\times M)$. In fact, we have a better estimate: 

\begin{lem}
\label{phij bound}
Let $\lambda_0 >0 $ be the lowest eigenvalue of $-\mathscr{A}'$. Then there exists a constant $C>0$ such that for all $t>1$, 
\[
|\mathbf{k}(t,p,q)| \leqslant C e^{- \lambda_0 t}. 
\] 
\end{lem}

\begin{proof}
Integrating $\varphi_j$ agains $\mathbf k(t, \cdot,\cdot)$ and using \eqref{density comparison} and \eqref{eq:decomp-k}, we see that for $s>0$:
\[
e^{-s\lambda_j}|\varphi_j(p)| \leqslant \left(\int_{M}| \mathbf p(s, p,q) |^2{\rm dvol}_g(q)\right)^{1/2}=C_s < \infty,
\]
due to Lemma \ref{Linfty estimate for density}. Coming back to \eqref{eq:decomp-k}, we observe that with $s = 1$,
\begin{align*}
|\mathbf{k}(t,p,q)| &\leqslant \sum_j e^{- \lambda_j t} |\varphi_j(p)||\varphi_j(q)| \\
					&\leqslant \sum_j e^{- \lambda_j (t-2s)} C_1^2 = C_1^2 \sum_j e^{- \lambda_j (t-2)}. 
\end{align*}
Since the last sum converges absolutely, it is $\mathcal{O}( e^{-\lambda_0 t})$ as $t\to+\infty$. 
%
%
%
\end{proof}


\begin{cor}
\label{ueps is L infinity}
For all $\varepsilon >0$, $u_\varepsilon \in L^\infty(M)$. 
\end{cor}

\begin{proof}
We split \eqref{eq:integral-representation} into two parts
$$u_\varepsilon(p) =  \int_0^1 \int_{\Omega_\varepsilon} {\mathbf k}(t,p,q){\rm dvol}_g(q)\dd t + \int_1^\infty \int_{\Omega_\varepsilon} {\mathbf k}(t,p,q){\rm dvol}_g(q)\dd t$$
The second integral can be estimated using Lemma \ref{phij bound}, while for the first one, we may use \eqref{density comparison} to get
$$
\left|\int_0^1 \int_{\Omega_\varepsilon} {\mathbf k}(t,p,q){\rm dvol}_g(q)\dd t \right| \leqslant  \left|\int_0^1 \int_{M} {\mathbf p}(t,p,q){\rm dvol}_g(q)\dd t\right| = 1,
$$
which concludes the proof.
\end{proof}

\subsubsection{Study of the generator}

\label{sssection:generator}

We now show the second part of Proposition \ref{proposition integral rep}, namely, that the generator $\mathscr{A}'$ of $T_t$ can be interpreted as $\mathbf{1}_{\Omega_\varepsilon} \mathscr{A}$ on sufficiently regular functions.
For this we will use the approximating semi-groups $(T_t^\ell)_{t \geqslant 0}$ introduced in \eqref{eq:def-semi-groups}. Originally in \cite{getoor1959markov}, they were defined as semi-groups on $L^2(M)$, but as we mentioned, the bound \eqref{eq:bounded-on-L1} on the kernels implies that they are actually bounded semi-groups on each $L^m(M)$, $m\in [1,+\infty]$. 
%
%
We have the analog of \cite[Theorem 2.1]{getoor1959markov}:

\begin{lem}
\label{strong continuity of semigroups}
For all $m\in [1,\infty)$, the semi-groups $(T_t)_{t \geqslant 0}$, $(U_t)_{t \geqslant 0}$, and $(T^\ell_t)_{t \geqslant 0}$ are strongly continuous on $L^m(\Omega_\varepsilon)$, $L^m(M)$, and $L^m(M)$ respectively.
\end{lem}

\begin{proof}
The proof of Geetor applies almost verbatim, replacing $L^2$ therein by $L^m$. Its main ingredient is the fact that for $\phi\in C^0(\Omega)$ (with the notations of \S \ref{sec:prelim-markov})
\[
T_t[V,\Omega]\phi(p) \to \phi(p),\ \text{ for a.e }\ p\in \overline{\Omega}.
\]
With boundedness on $L^\infty$ and dominated convergence, this implies that $(T_t)_{t \geqslant 0}$ is weakly continuous on all $L^m$ with $m< \infty$, and hence strongly continuous.
\end{proof}
As the semigroups on $L^m$ are strongly continuous, they all admit generators which are densely defined but the question then is to characterize their domain. We first give a precise description of the domain of the infinitesimal generator $\mathscr{A}$ of the semigroup $(U_t)_{t \geqslant 0}$ given by \eqref{transition probability p} on $L^m(M)$. Recall that the domain of the generator of the semigroup $U_t :L^m(M) \to L^m(M)$, which we denote by $D_{L^m(M)}(\mathscr{A})$, is defined by 
\[
D_{L^m(M)}(\mathscr{A}) := \{u\in L^m(M) \mid \lim\limits_{t\to 0}  \left( U_tu - u\right)/t \ {\text{ converges in }} L^m(M) \}.
\]
In what follows, when the $L^m(M)$ space is clear within the context, we will drop the $L^m(M)$ subscript and simply write $D(\mathscr{A})$.
\begin{lem}
\label{domain is fractional sobolev}
For each $m\in (1,\infty)$, let $U_t : L^m(M) \to L^m(M)$ be the strongly continuous semigroup given by Lemma \ref{strong continuity of semigroups}. Then $D(\mathscr{A}) = W^{2\alpha, m}(M)$. 
\end{lem} 

\begin{proof}
According to \cite{AppEst} (and \cite{App-95} as mentioned in the footnote of \S\ref{sec:prelim-markov}), the domain of the generator contains $C^\infty(M)$. Since the semi-group $U_t$ is strongly continuous and contracting, $(\mathscr{A},D(\mathscr{A}))$ must be closed. However, since $\mathscr{A}$ is elliptic, $(\mathscr{A},C^\infty(M))$ has only one closure in $L^m$, $m\in (1,+\infty)$ which must be $(\mathscr{A}, W^{2\alpha,m})$, see \cite[Theorem 12.15]{Wong-book} for a reference. 
\end{proof}

The next step will be to show that the generator $\mathscr{A}'$ of the semigroup $T_t:  L^m( \Omega_\varepsilon) \to L^m( \Omega_\varepsilon)$ defined in \eqref{eq:def-semi-groups} coincides with $\1_{\Omega_\varepsilon}\mathscr{A}$:
\begin{lem}
\label{restricted semigroup generator}
For $m\in [1,+\infty)$, we have $L^m(\Omega_\varepsilon)\cap D(\mathscr{A}) \subset D(\mathscr{A}')$ with $\mathscr{A}' u = \1_{\Omega_\varepsilon} \mathscr{A}u$ for all $u\in L^m(\Omega_\varepsilon)\cap D(\mathscr{A})$.
\end{lem}
Of course, this is only useful for $m>1$, in which case we are able to characterize $D(\mathscr{A})$. The only ingredient in the proof is the strong continuity of $(U_t)_{t \geqslant 0}$ on $C^0(M)$.

\begin{proof} We proceed as in \cite{getoor1961first}.
First, \cite[Theorem 5.1]{getoor1957additive} asserts that if $\mathscr{A}_\ell$ is the infinitesimal generator of $(T^\ell_t)_{t \geqslant 0}$, then 
\begin{eqnarray}
\label{Aell}
\mathscr{A}_\ell = \mathscr{A} - \ell\1_{B_\varepsilon(p_0)}.
\end{eqnarray}
with domain $D(\mathscr{A}_\ell) = D(A)$. Let $(\mathscr{A} - \lambda)^{-1}$, $(\mathscr{A}' - \lambda)^{-1}$, and $(\mathscr{A}_{\ell} - \lambda)^{-1}$ be resolvents for the semigroups $U_t$, $T_t$, and $T_t^\ell$ respectively. Assume that $\lambda>0$, and write
\[
(\lambda-\mathscr{A}')^{-1} = \int_0^{+\infty} e^{-t\lambda} T_t dt,\quad (\lambda-\mathscr{A}_\ell)^{-1} = \int_0^{+\infty} e^{-t\lambda} T_t^\ell dt. 
\]
According to \eqref{eq:approximating-semi-group}, for $\phi\in L^2(\Omega_\varepsilon)$, $T_t^\ell \phi \to T_t \phi,\ \text{ in } L^2(M)$. As we have already observed in \S \ref{sec:prelim-markov}, this result extends to $\phi\in L^m(\Omega_\varepsilon)$ with convergence in $L^m(M)$. We will need a stronger statement, namely, that this holds for all $\phi \in L^m(M)$ with convergence in $L^m(M)$ to $T_t (\mathbf{1}_{\Omega_{\varepsilon}}\phi)$. Hence, it suffices to show that if $\phi \in L^m(M \setminus \Omega_{\varepsilon})$, then $T_t^\ell \phi \to 0$ in $L^m(M)$. In the proof of \cite[Theorem 4.1]{getoor1959markov}, we see that for $\phi\in C^0(M)$, we have the pointwise convergence
\[
\lim_{\ell \to \infty} T_t^\ell \phi(p) = \mathbb{E}\left[ \phi(X(t)) \1_{\{X(\tau)\in \Omega_\varepsilon | 0\leqslant \tau < t\}} \middle| ~X(0)=p\right].
\]
Hence, if $\phi$ is continuous and supported in $M \setminus \Omega_\varepsilon$, we have the pointwise convergence $T_t^\ell \phi(p) \to 0$ (by dominated convergence). {Then, using that $T_t^{\ell}$ is a contraction on $L^m(M)$ and approximating a function $\phi \in L^m(M \setminus \Omega_{\varepsilon})$ by continuous functions $\phi_n \to \phi$ in $L^m(M \setminus \Omega_{\varepsilon})$, we deduce that $T_t^\ell \phi \to 0$ in $L^m(M)$.
This shows that for all $\phi \in L^m(M)$, $T_t^\ell \phi \to_{\ell \to \infty} T_t (\1_{\Omega_\varepsilon}\phi)$ with convergence in $L^m(M)$.}

Now, since $\lambda>0$, we can use dominated convergence to conclude that for all $L^m(M)$ with $m< \infty$, $\phi\in L^m(M)$
\begin{equation}\label{convergence of semigroup resolvent}
(\lambda-\mathscr{A}_\ell)^{-1} \phi \to (\lambda-\mathscr{A}')^{-1} \1_{\Omega_\varepsilon} \phi,\ \text{ in $L^m(M)$.}
\end{equation}

For $u\in D(\mathscr{A})\cap L^m(\Omega_\varepsilon)$, we observe that
\[
\mathscr{A}_\ell u = (\mathscr{A} - \ell \1_{B_\varepsilon(p_0)}) u = \mathscr{A} u. 
\]
In particular, 
\[
u = (\lambda-\mathscr{A}_\ell)^{-1}(\lambda-\mathscr{A})u \to (\lambda - \mathscr{A}')^{-1} \1_{\Omega_\varepsilon}(\lambda - \mathscr{A})u. 
\]
so that $u$ is in $D(\mathscr{A}')$ and
\[
(\lambda -\mathscr{A}') u = \lambda u - \1_{\Omega_\varepsilon} \mathscr{A} u. 
\]
\end{proof}

The previous lemma completes the proof of Proposition \ref{proposition integral rep}. In turn, we can now conclude the proof of Proposition \ref{proposition uniqueness}:

\begin{proof}
Let $w \in \dot{W}^{2\alpha,m}(\overline{\Omega}_\varepsilon)$ such that $\mathscr{A} w = -1$ in $\Omega_\varepsilon$. This means that $\mathscr{A}' w = - \mathbf{1}_{\Omega_\varepsilon}$. But using \eqref{eq:integral-representation}, we get the following equality in $L^m(\Omega_\varepsilon)$:
\[
u_\varepsilon = \int_0^{+\infty} T_t \mathbf{1}_{\Omega_\varepsilon}\dd t = - \int_0^{+\infty} T_t \mathscr{A}' w\dd t = - \int_0^{+\infty} \partial_t (T_t w)\dd t = w
\]
since $T_t w \to_{t \to \infty} 0$ in $L^\infty(\Omega_\varepsilon)$ by Lemma \ref{phij bound}. Moreover, $u_\varepsilon = 0 = w$ in $B_\varepsilon(p_0)$. Since $u_\varepsilon \in L^\infty(M)$ and $w \in L^m(M)$, this implies that $u_\varepsilon = w$ over $M$. 
\end{proof}

\bibliographystyle{alpha}
\bibliography{biblio}
\end{document}